\documentclass{article}
\usepackage{pinlabel}
\usepackage{graphicx}
\usepackage{amsmath}
\usepackage{amssymb}
\usepackage{amsthm}
\usepackage{multirow}
\usepackage{color}
\usepackage{hyperref}
\usepackage[english]{babel}
\usepackage{epstopdf}   
\usepackage{cancel}
\usepackage[dvipsnames]{xcolor}

\input epsf.tex
\newcommand{\pic}[2]{\raisebox{-.5\height}
{\includegraphics[scale=#2]{#1}}}

\def\unknot{\pic{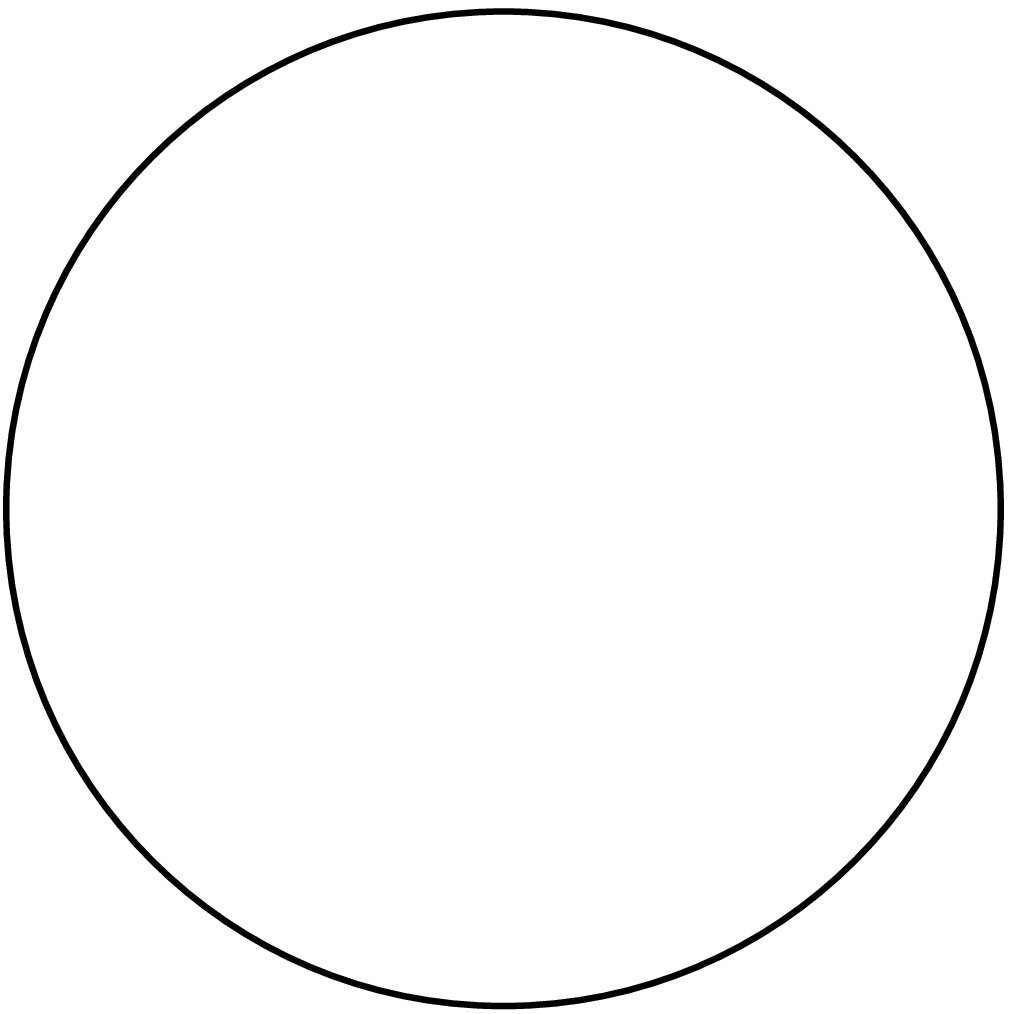}{.025}}

\def\cruce{\pic{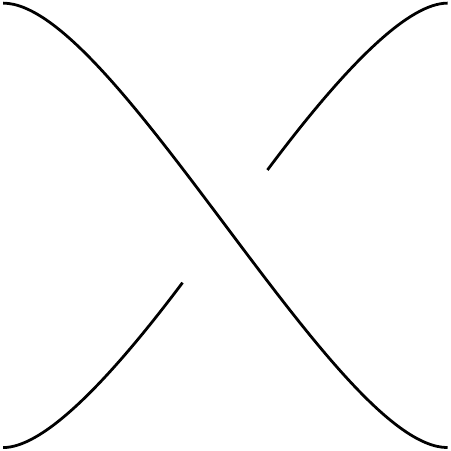}{.055}}
\def\cupcap{\pic{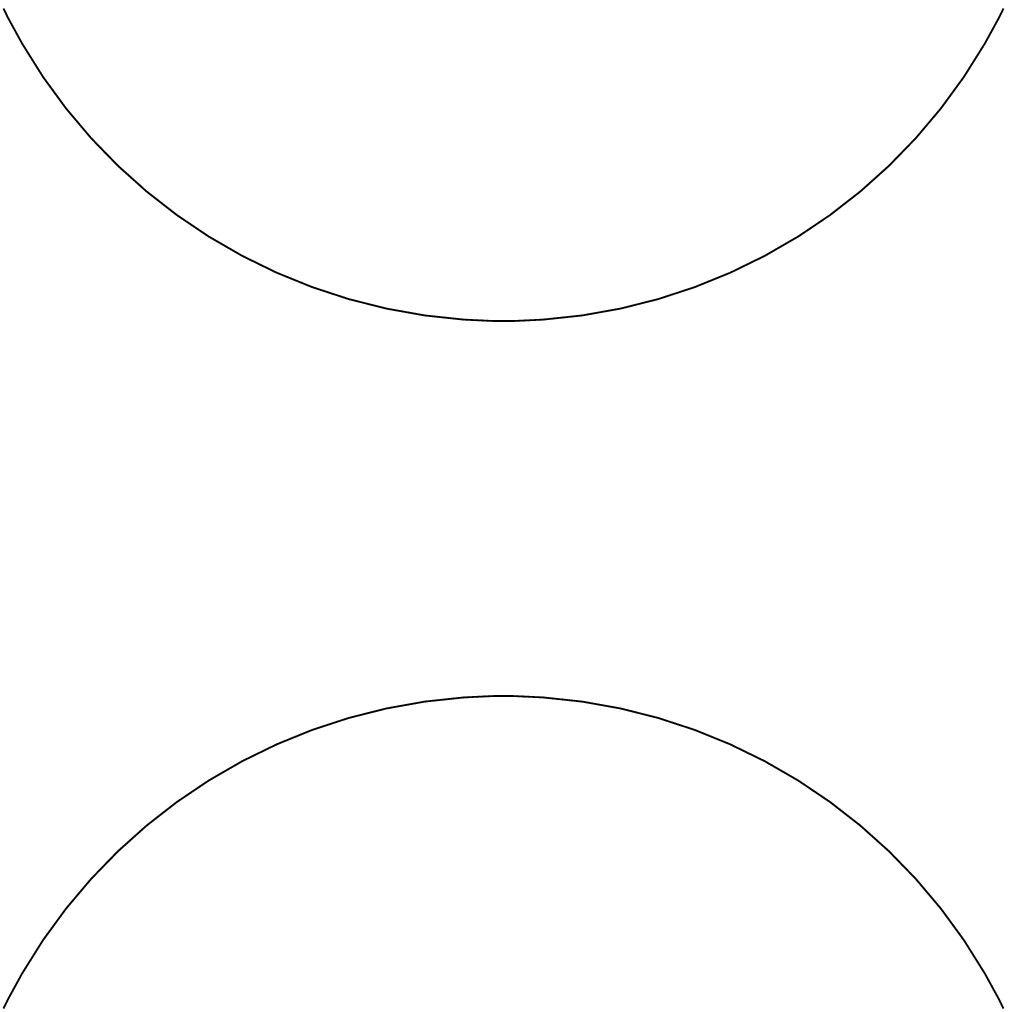}{.025}}

\def\line{\pic{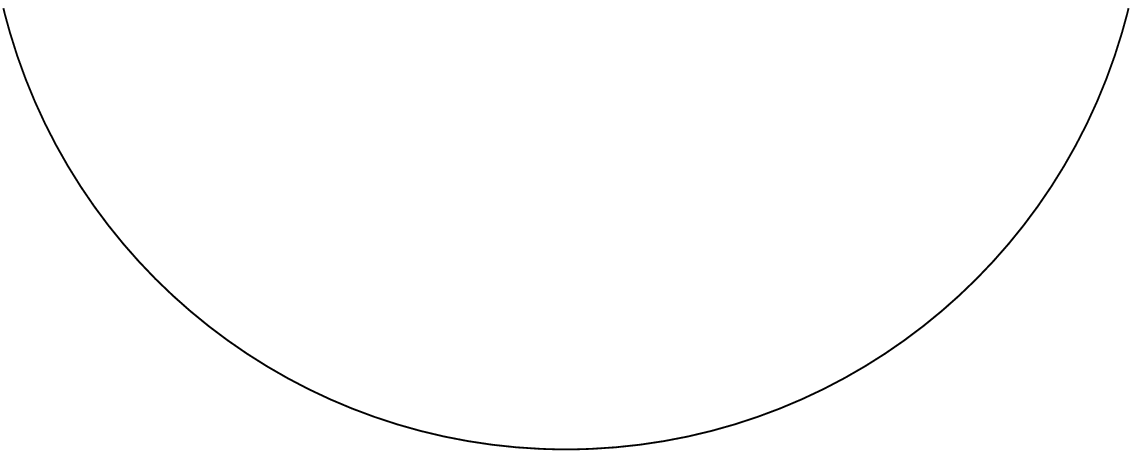}{.025}}
\def\negative{\pic{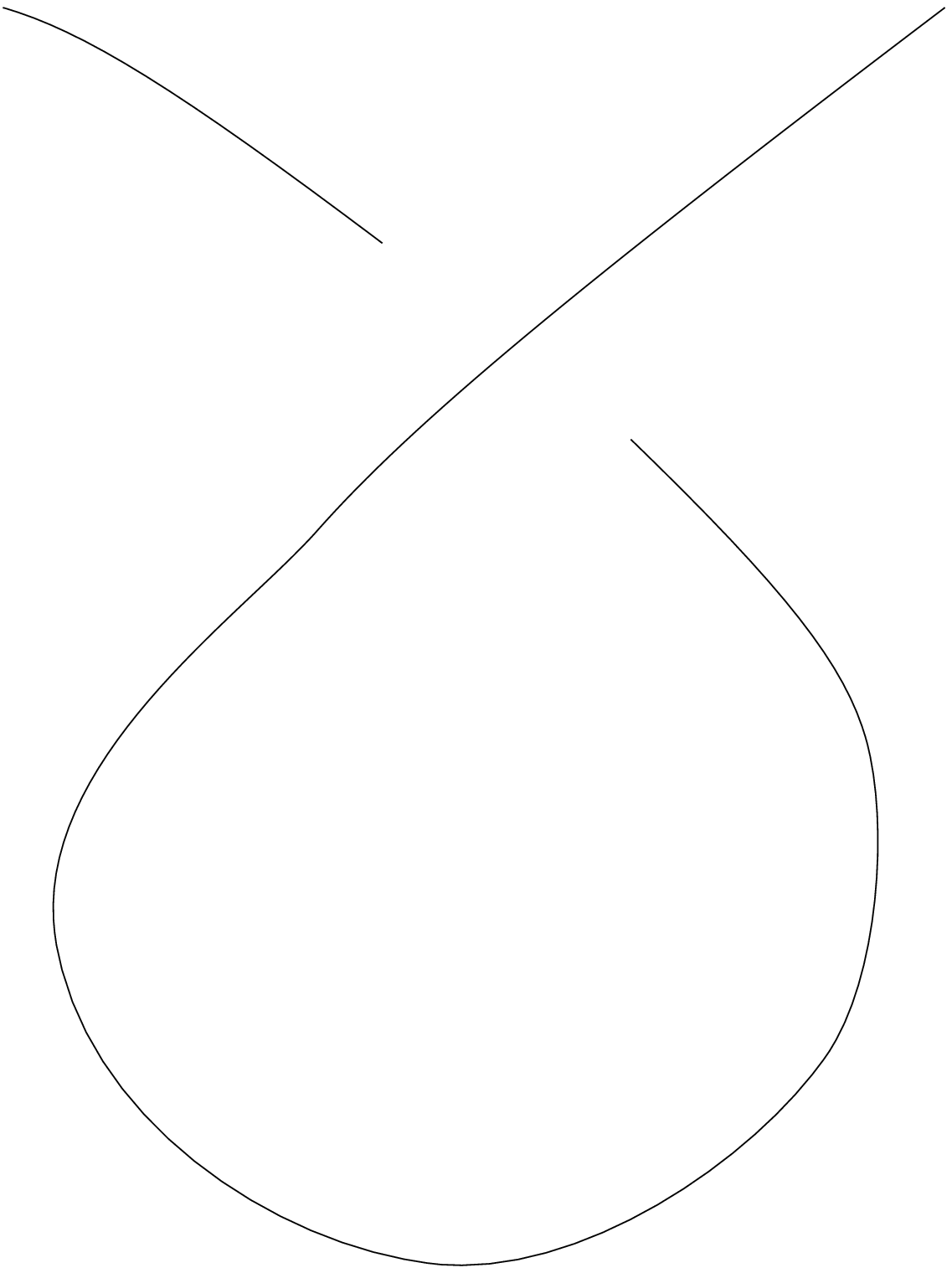}{.025}}
\def\positive{\pic{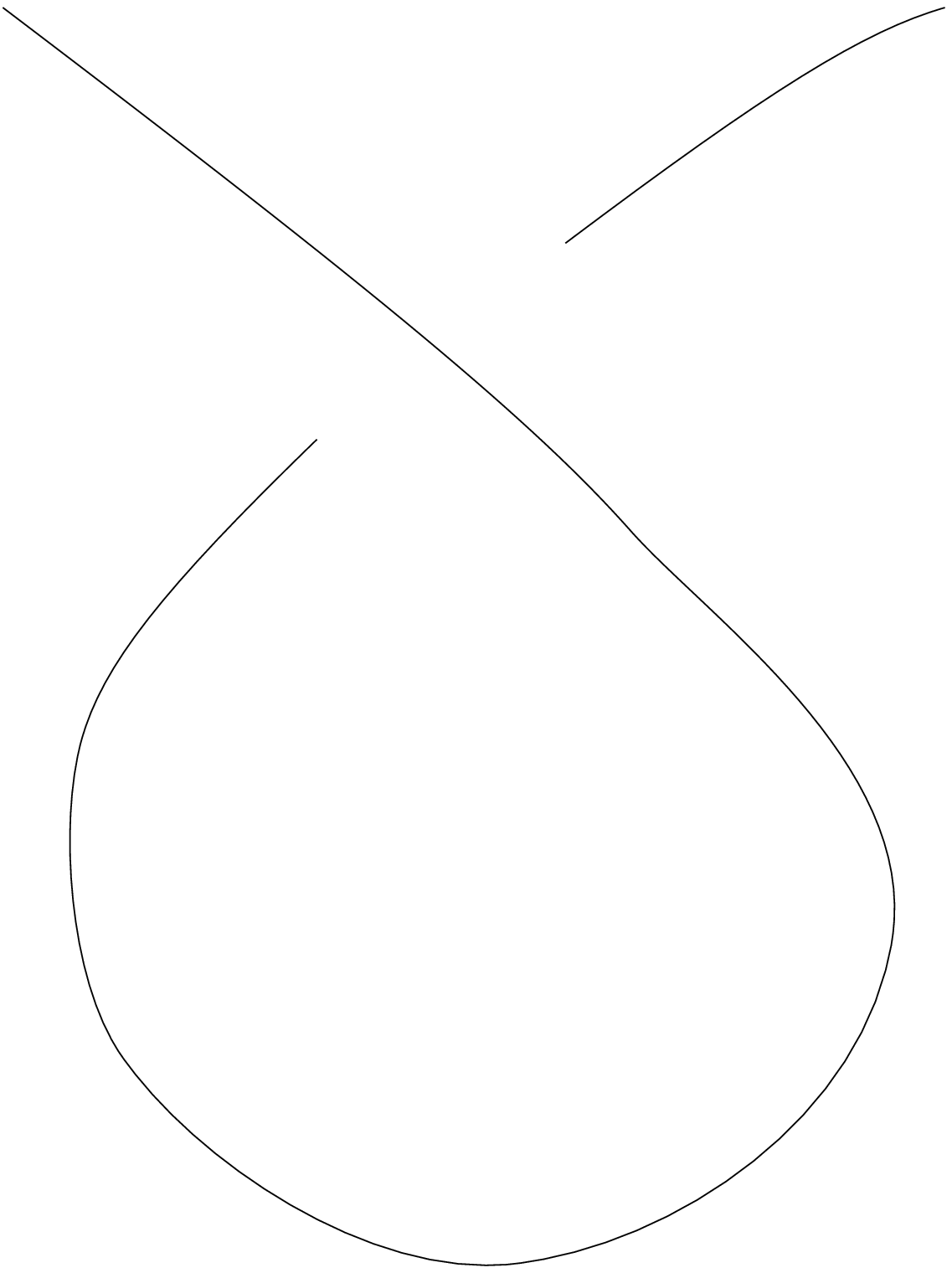}{.025}}
\def\parentesis{\pic{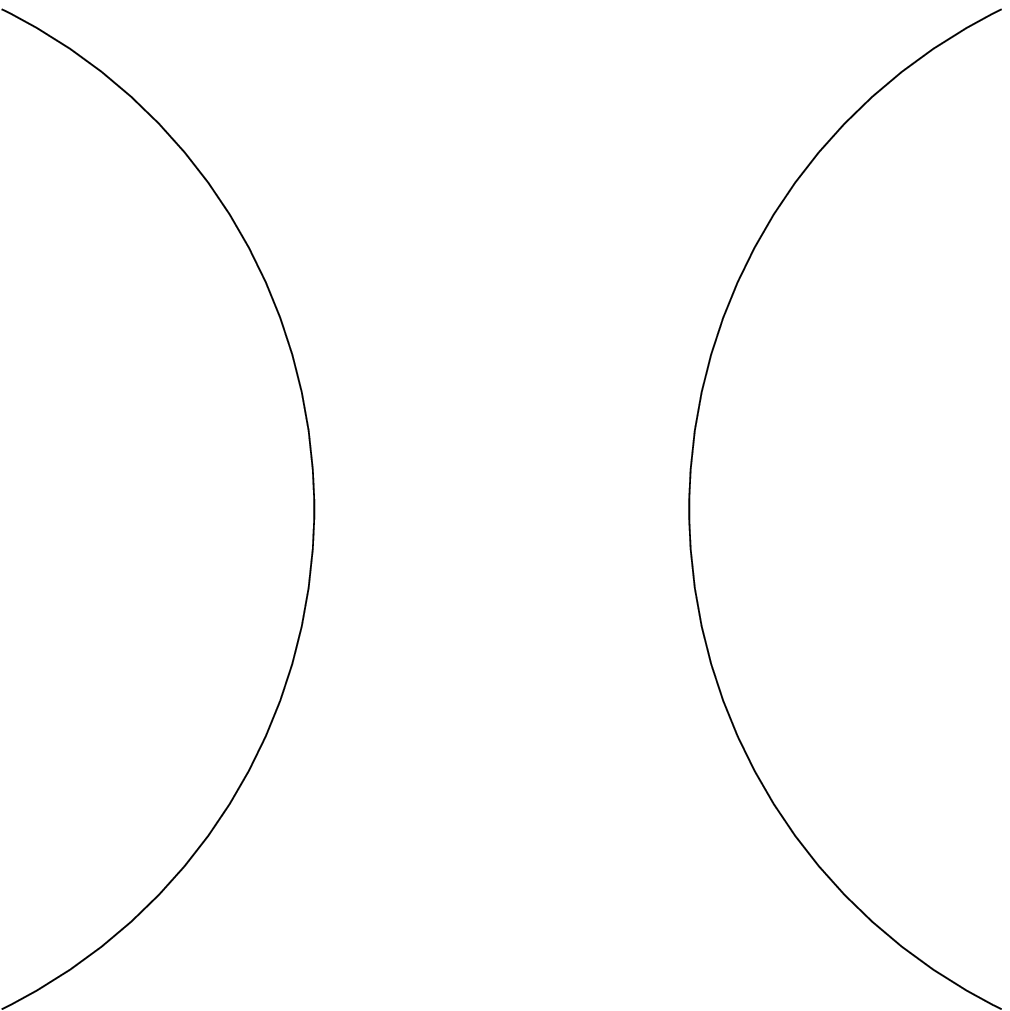}{.025}}
\def\pretzel{\pic{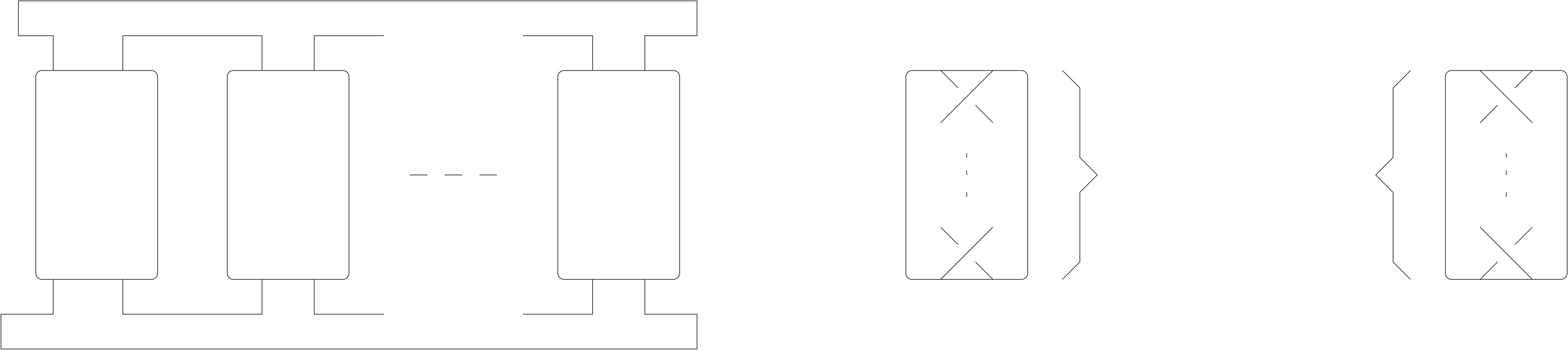}{.150}}
\def\politopo{\pic{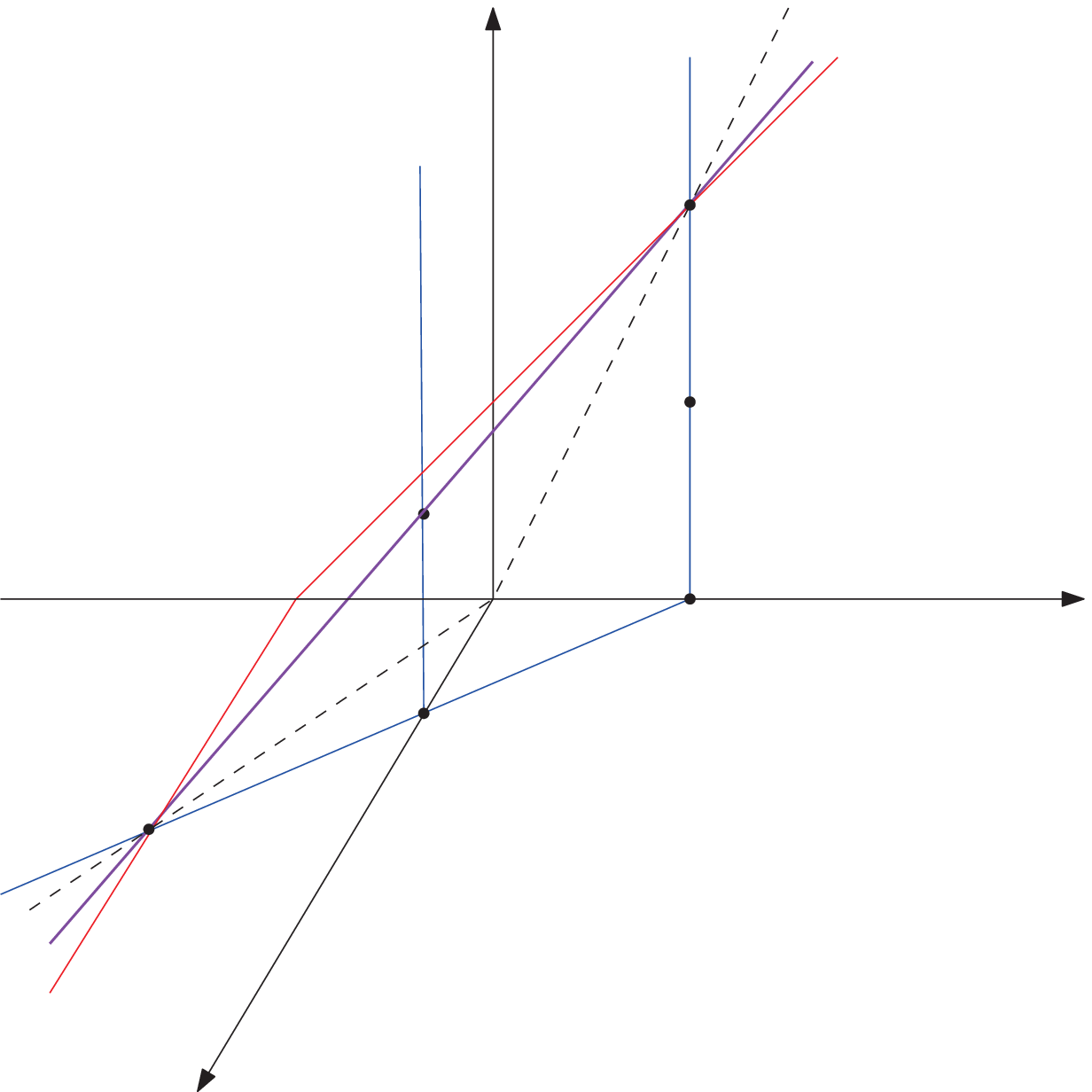}{.900}}
\def\plano{\pic{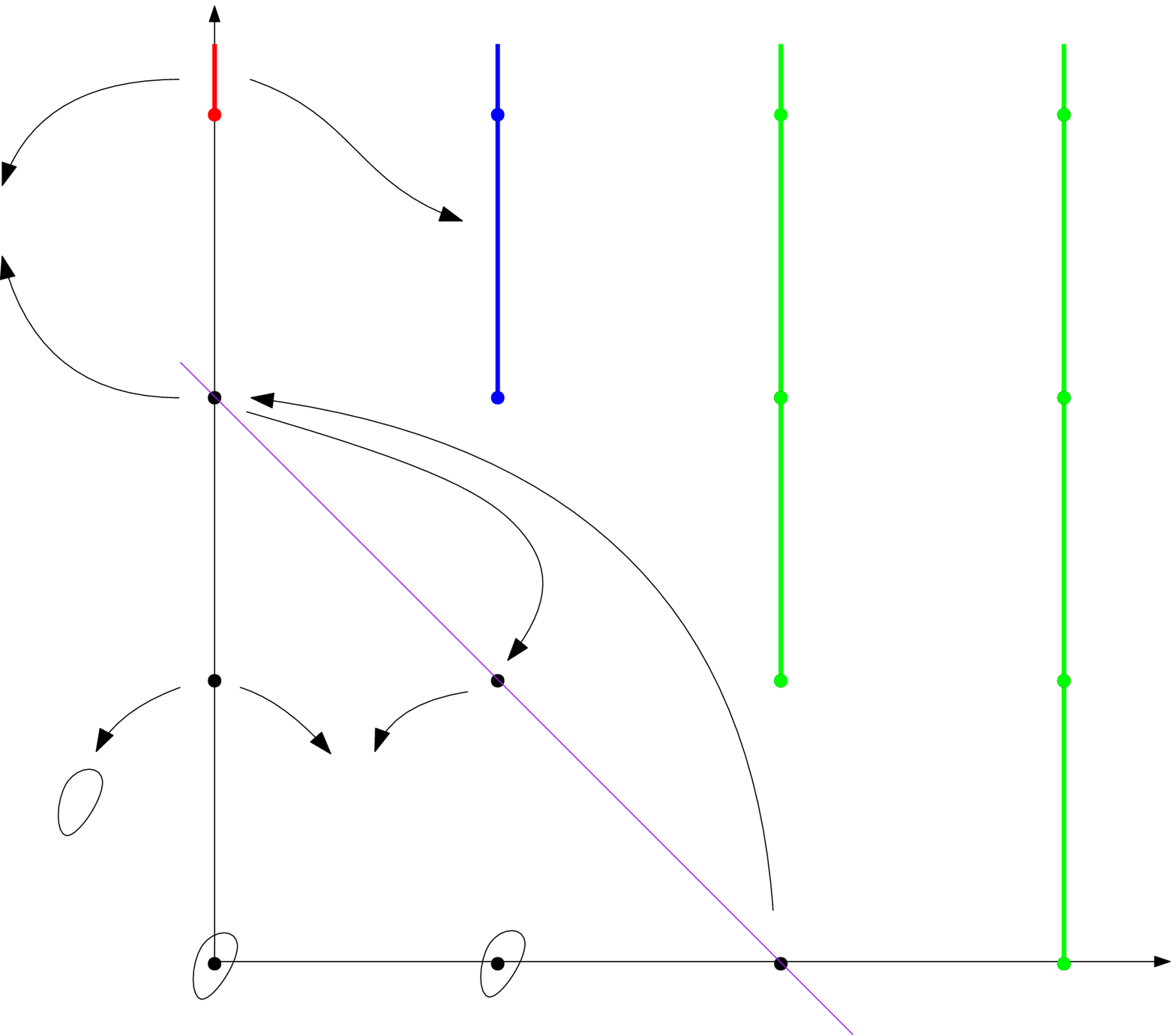}{.580}}
\def\tresuno{\pic{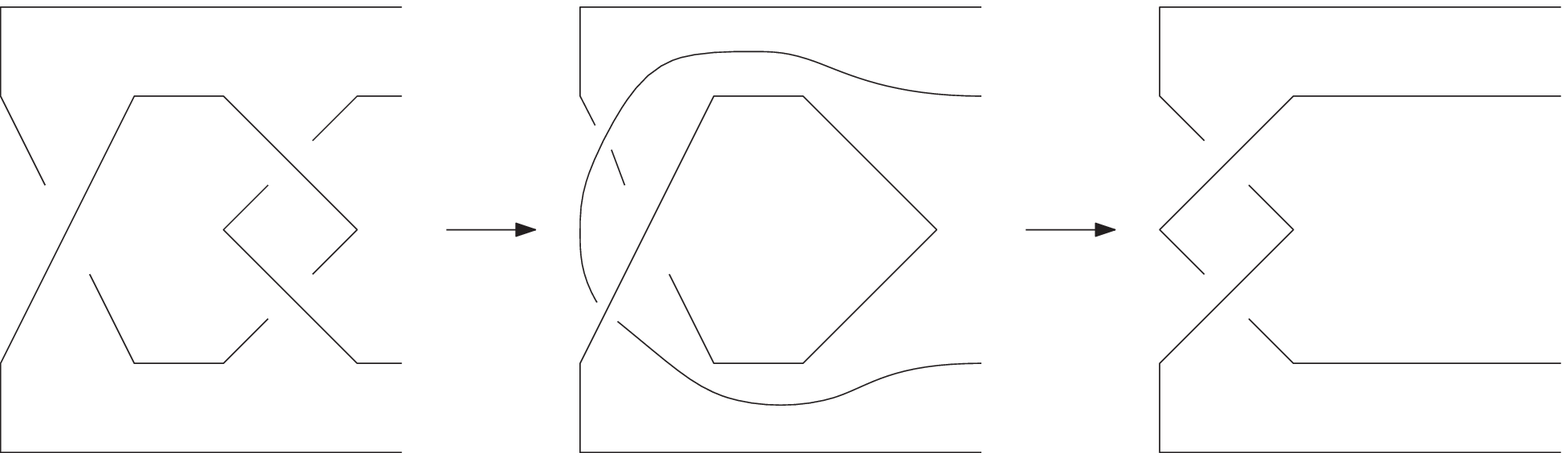}{.350}}
\def\unomenosuno{\pic{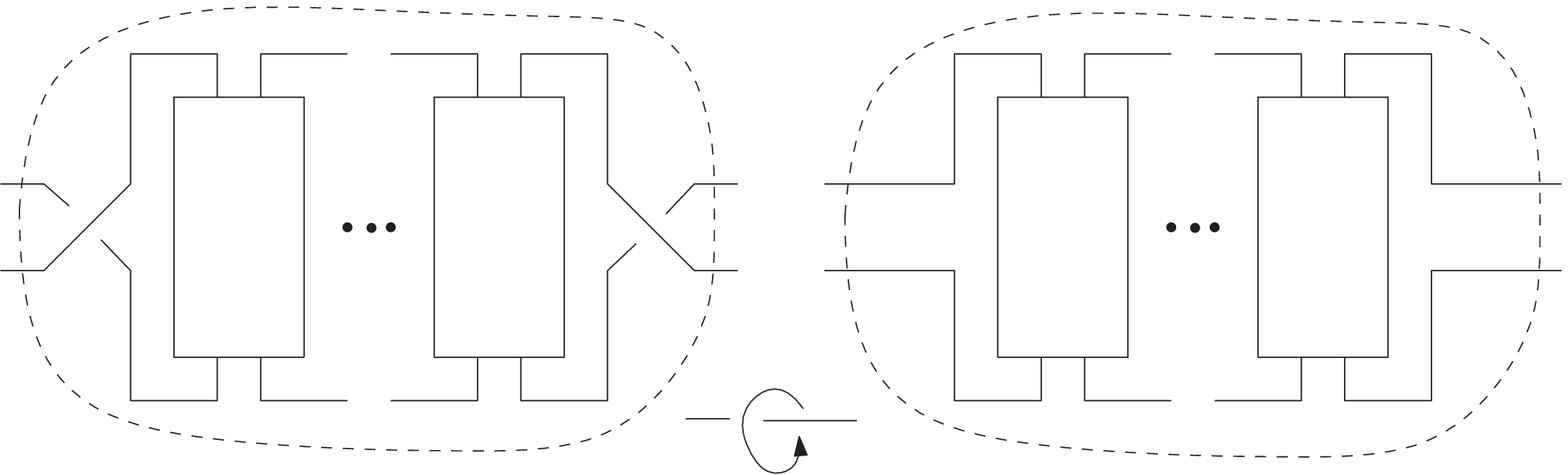}{.350}}

\def\OVUpretzelUno{\pic{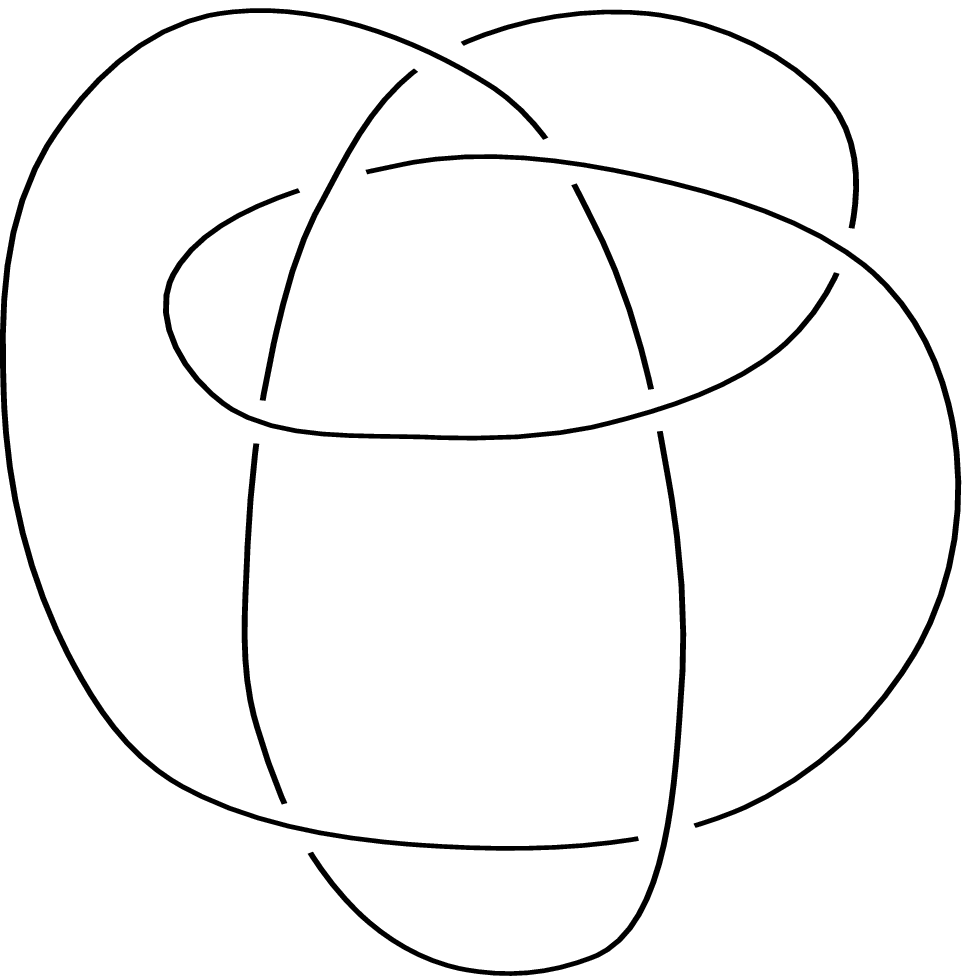}{.200}}
\def\OVUpretzelDos{\pic{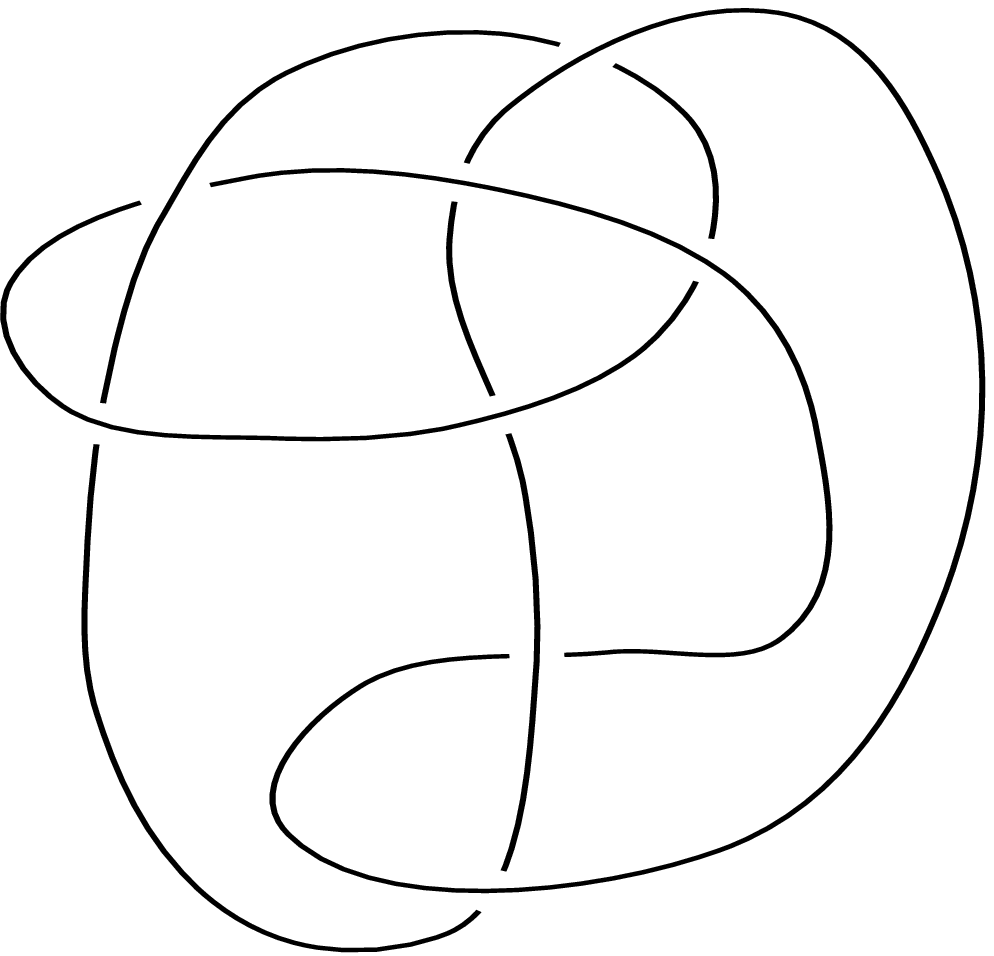}{.200}}
\def\OVUpretzelTres{\pic{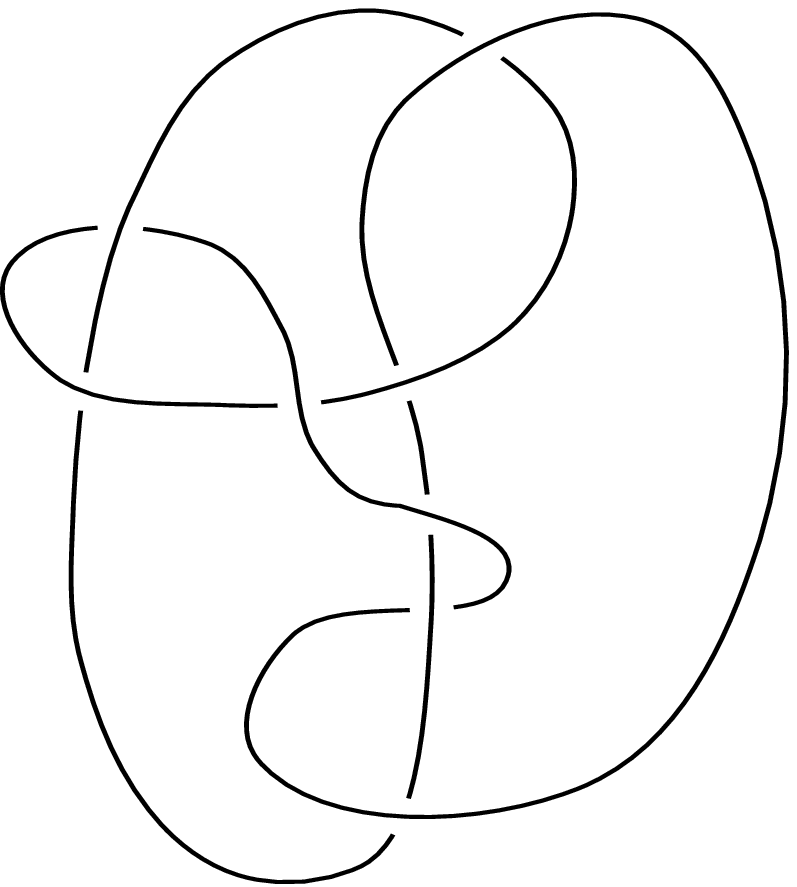}{.200}}
\def\OVUpretzelCuatro{\pic{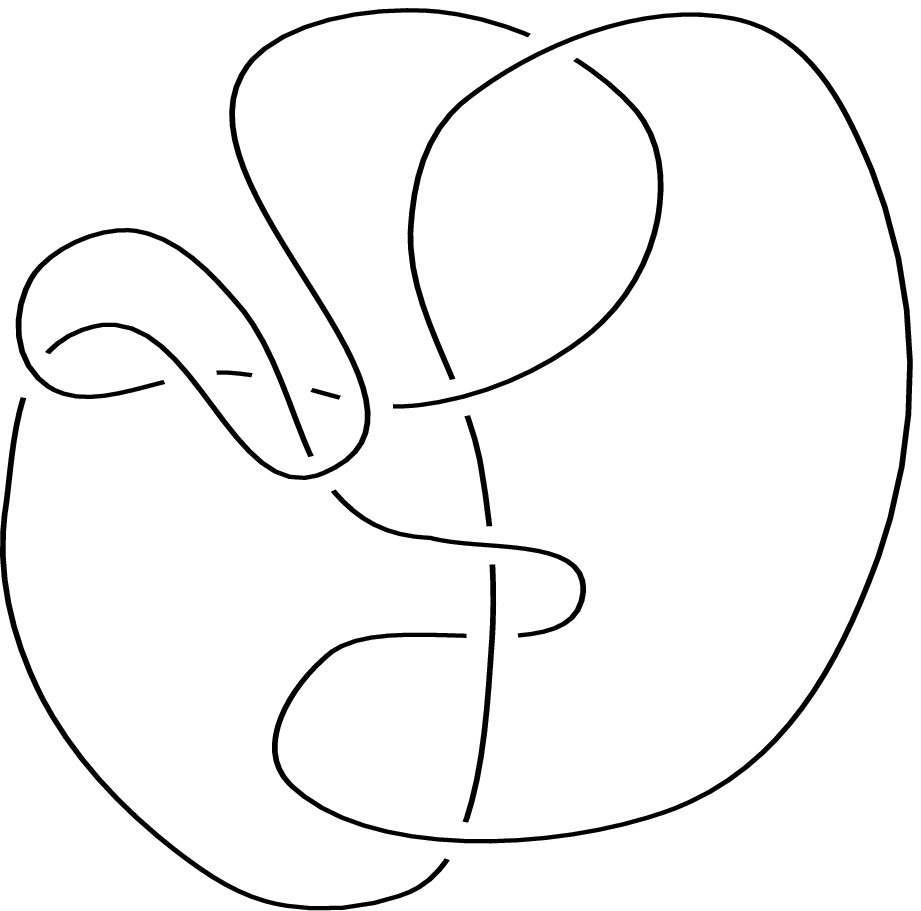}{.200}}
\def\OVUpretzelCinco{\pic{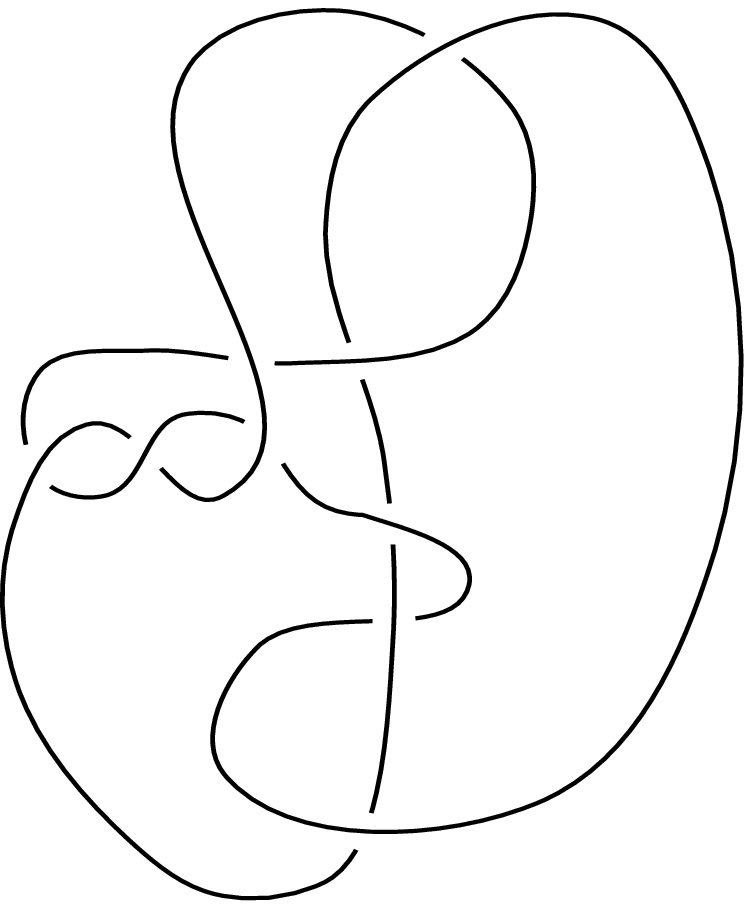}{.200}}
\def\OVUpretzelSeis{\pic{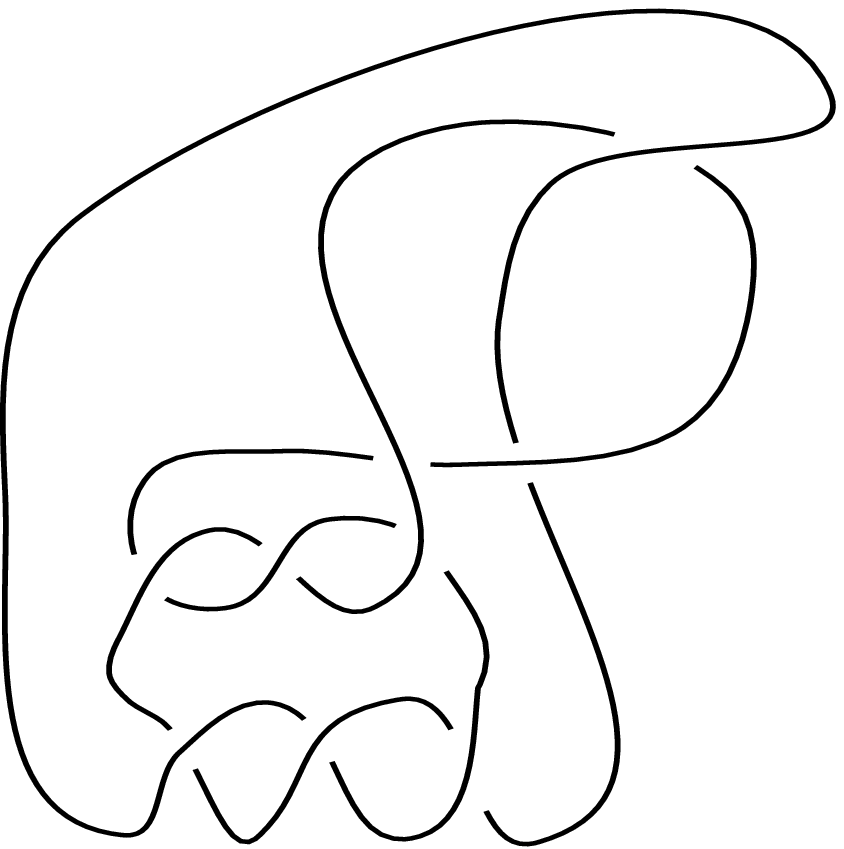}{.200}}
\def\OVUpretzelSiete{\pic{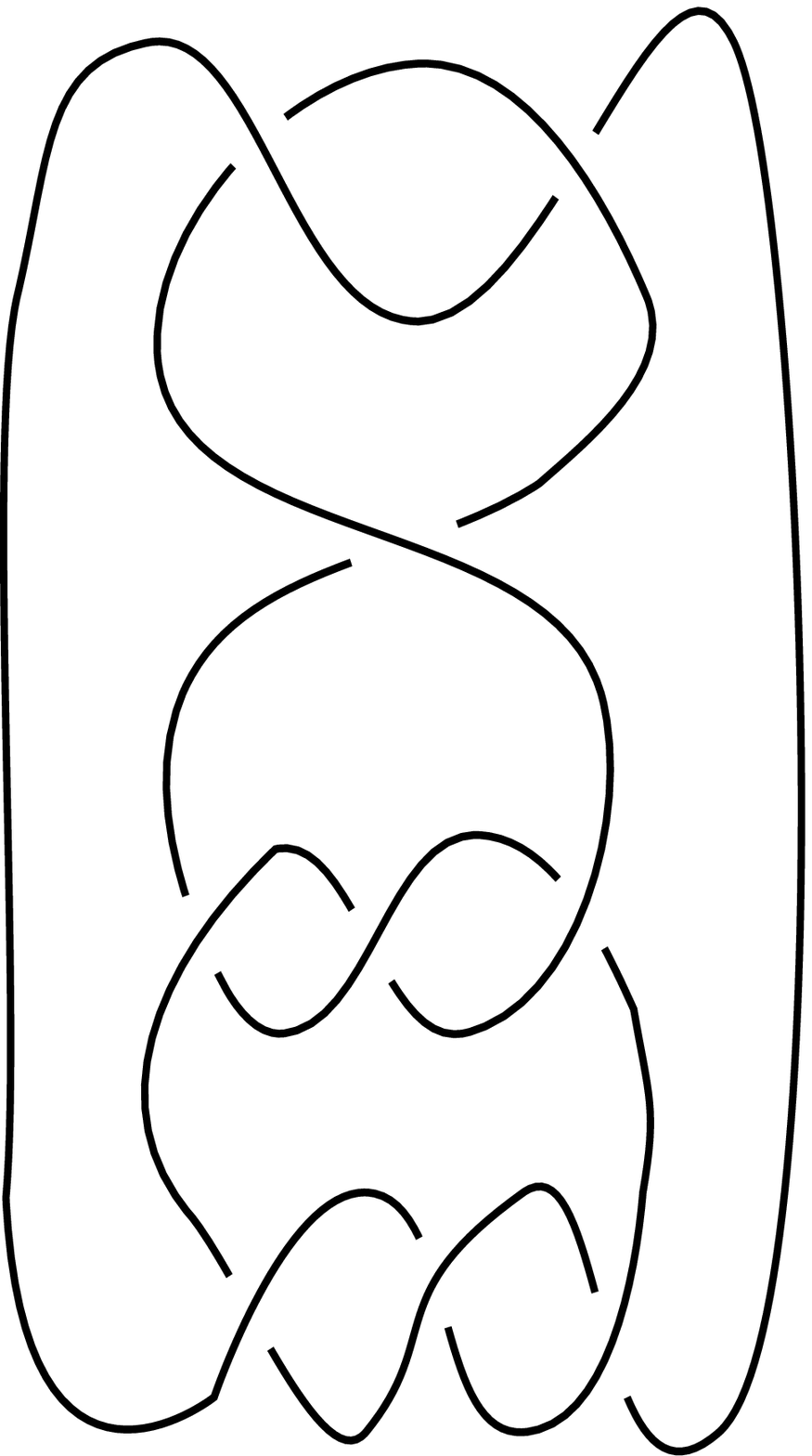}{.200}}

\newtheorem{theorem}{Theorem}
\newtheorem{proposition}[theorem]{Proposition}
\newtheorem{corollary}[theorem]{Corollary}
\newtheorem{lemma}[theorem]{Lemma}
\newtheorem{definition}[theorem]{Definition}

\newtheorem{remark}[theorem]{Remark}
\newtheorem{notation}[theorem]{Notation}

\renewenvironment{proof}[1][Proof]{\textit{#1.} }
{\hfill \rule{0.5em}{0.5em}}

\newcommand{\tn}{\textnormal}

\begin{document}

\title{Pretzel knots up to nine crossings}

\author{R. D\'{\i}az and P. M. G. Manch\'on
\footnote{The first author is partially supported by Project MTM2017-89420-P. The second author is partially supported by MEC-FEDER grant MTM2016-76453-C2-1-P.}}
\maketitle

\begin{center}
{\it \begin{tabular}{c} The first author dedicates this paper to her father, \\ who helped her with showing that the knot $7_7$ is pretzel.\end{tabular}}
\end{center}

\begin{abstract}
There are infinitely many pretzel links with the same Alexander polynomial (actually with trivial Alexander polynomial). By contrast, in this note we revisit the Jones polynomial of pretzel links and prove that, given a natural number $S$, there is only a finite number of pretzel links whose Jones polynomials have span $S$.

More concretely, we provide an algorithm useful for deciding whether or not a given knot is pretzel. As an application we identify all the pretzel knots up to nine crossings, proving in particular that $8_{12}$ is the {\it first} non-pretzel knot.
\end{abstract}

\textbf{Keywords:} \emph{Pretzel link, Kauffman bracket, Jones polynomial, span.}

\textbf{MSC Class:} \emph{57M25, 57M27.}

\section{Introduction} \label{SectionIntroduction}

The original motivation of this paper was to complete some tables appearing in some books (and knot atlas)  deciding when a knot is or not of type pretzel (see for example \cite{Peter}).  
In a previous note \cite{Pedro} by the second author, it was given a closed formula for the Kauffman bracket of any pretzel diagram $P=P(a_1, \dots, a_n)$ (a recurrence formula was given in \cite{Landvoy}), and based on this formula, the span of the Jones polynomial of the pretzel link represented by $P$ in terms of its entries $a_1, \dots, a_n$ (see Theorem~\ref{TheoremMain} and Theorem~\ref{TheoremSpan}). In this paper, which can be seen as a natural continuation of \cite{Pedro}, we prove the following result (Theorem \ref{TheoremFewCases}): given an integer $S$, there is a finite number of pretzel links whose Jones polynomials have span $S$. Moreover, the complete list of the pretzel links with span $S$ can be provided by using Theorem \ref{TheoremSpan}. 

A by-product of our work is to provide the complete list of pretzel knots up to nine crossings (Theorems\ref{TheoremUpToEightCrossings} and \ref{TheoremNineCrossings}). In particular, we discover that $8_{12}$ is the first non-pretzel knot in tables, and $9_{42}$ is the first non alternating and non pretzel knot. 

The list of items in the statement of Theorem \ref{TheoremSpan} is long and a bit cumbersome, and in this paper we fix a small exception that was missing in \cite{Pedro}. For this reason and by completeness, we rewrite in an appendix the proof of Theorem~ \ref{TheoremSpan}, taking special care of item {\it 5}.

Pretzel links (and rational links) are special types of Montesinos links, and there have been some important works trying to classify Montesinos links up to mutation and $5$-move (see for example \cite{DIP} or \cite{Stoimenow}). Sub-products of these works have been certain (partial) tests for trying to decide whether or not a link is of Montesinos type, and the same for pretzel links. For example, if the test in \cite{Stoimenow} says {\it no}, the link is not a Montesinos link; if it says {\it yes}, the link {\it could be} of type Montesinos, and if so, would be obtained by applying mutations and $5$-moves to an specific (representative of an equivalence class of) Montesinos link. Our contribution to this subject goes in another direction. When we face an arbitrary link $L$, our algorithm finds a (usually large but) {\it finite} list of pretzel links such that, if $L$ is pretzel, must be one of the list, the key points being Theorems \ref{TheoremFewCases} and \ref{TheoremSpan}. 

The paper is organized as follows: in Section \ref{SectionAgain} we recall the basic definitions and the closed formula for the Kauffman bracket of pretzel links. In Section \ref{SectionSpanJones} we state Theorem \ref{TheoremSpan} and discuss its items with the help of some pictures. In Section \ref{SectionFewCases} we prove Theorem \ref{TheoremFewCases} and precise an upper bound for the number of pretzel links with a given span. Section \ref{SectionList} uses carefully the previous theorems for deciding which knots up to nine crossings are pretzel. The revisited proof of Theorem \ref{TheoremSpan} is left to the appendix. 

\section{Kauffman bracket of pretzel links} \label{SectionAgain}
Given integers $a_1,...,a_n$, denote by $P(a_{1},...,a_{n})$ the pretzel link diagram shown in Figure \ref{FigurePretzel}. Here $a_i$ indicates $|a_i|$ crossings, with signs $a_i/|a_i|$ if $a_i\neq 0$.
\begin{figure}[ht!]
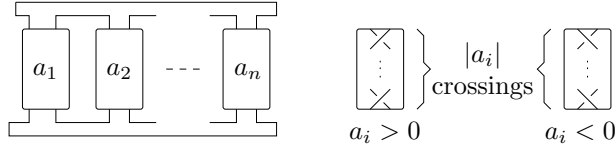
 
\labellist
 \pinlabel {$a_1$} at 100 175
 \pinlabel {$a_2$} at 300 175
 \pinlabel {$a_n$} at 650 175
 \pinlabel {$a_i>0$} at 1020 10
 \pinlabel {$a_i<0$} at 1550 10
 \pinlabel {\begin{tabular}{c}$|a_i|$\\ crossings\end{tabular}} at 1285 175
\endlabellist
\begin{center}
\pretzel
\end{center}
\caption{Pretzel link diagram $P(a_{1},...,a_{n})$.}
\label{FigurePretzel}
\end{figure}

A pretzel link is a link that has a pretzel diagram. 

For a link diagram $D$ we denote by $\langle D\rangle $ its Kauffman bracket with normalization $\langle \put(0,3){\unknot} \hspace{0.25cm} \rangle =\delta =-A^{-2}-A^{2}$ (see \cite{Lickorish}). Recall that $\langle D\rangle $ is a regular isotopy invariant of diagrams, defined by the following additional relations:
$$
\tn{(i)}\, \, \, 
\langle \put(0,3){\cruce} \hspace{0.25cm} \rangle 
= A\langle \put(0,3){\cupcap} \hspace{0.25cm}
\rangle +A^{-1}\langle \put(0,3){\parentesis} \hspace{0.25cm} \rangle, 
\qquad
\tn{(ii)}\, \, \,
\langle D \sqcup \put(0,3){\unknot} \hspace{0.25cm}
 \rangle =\delta \langle D \rangle.
 $$

Here $\put(0,3){\unknot } \hspace{0.25cm}$ is the diagram of the unknot with no crossings. In (i) the formula refers to three link diagrams that are exactly the same except near a point where they differ in the way indicated. In (ii) $D \sqcup \put(0,3){\unknot } \hspace{0.25cm}$ is a diagram consisting of the diagram $D$ together with an extra closed curve $\put(0,3){\unknot } \hspace{0.25cm}$ that contains no crossings at all, not with itself nor with $D$. From these relations we can deduce the effect on $\langle D\rangle $ of a type I Reidemeister move on $D$:
$$
\tn{(iii)} \, \, \,
\langle \put(0,3){\positive } \hspace{0.25cm} \rangle 
= -A^3\langle \put(0,3){\line } \hspace{0.25cm} \rangle,
\qquad
\tn{(iii')}\, \, \, 
\langle \put(0,3){\negative } \hspace{0.25cm} \rangle
=-A^{-3}\langle \put(0,3){\line } \hspace{0.25cm} \rangle.
$$

\begin{theorem} \label{TheoremMain} {\rm \cite{Pedro}} The Kauffman bracket of the
pretzel link diagram $P(a_{1},...,a_{n})$ is given by the formula
\begin{equation*}
\langle P(a_{1},\dots ,a_{n})\rangle =\prod_{i=1}^{n}(A^{a_{i}}\delta
+[a_{i}])+(\delta ^{2}-1)\prod_{i=1}^{n}[a_{i}].
\end{equation*}
\end{theorem}

Here $[0] = 0$, and for any integer $a\neq 0$, 
$$
[a] 
= A^{-2\frac{a}{|a|}-3a} \sum_{i=1}^{|a|}(-1)^{a+i}A^{4i\frac{a}{|a|}}.
$$ 
Note that $[a]$ is a Laurent polynomial in the variable $A$. Its behaviour reminds the quantum integers. Clearly $[1]=A^{-1}$,
$[-1]=A$, and we have the recurrence formulas 
\begin{center}
$[a]=A[a-1]+A^{-1}(-A^{-3})^{a-1}$ if $a>0$, \\
$[a]=A^{-1}[a+1]+A(-A^{3})^{-a-1}$ if $a<0$.
\end{center}

In addition we have the following equalities (proofs can be found in \cite{Pedro}):
\begin{lemma} \label{lemma 0} 
$\delta [a]=-A^{a}+(-A^{-3})^{a}$.
\end{lemma}

\begin{lemma} \label{lemma 1}
If the number of entries is greater than one, then
$$
\langle P (...,a_{i-1},a,a_{i+1},...)\rangle 
= A^{a}\! \langle P(...,a_{i-1},0,a_{i+1},...)\rangle 
+  [a] \langle P(...,a_{i-1},a_{i+1},...)\rangle.
$$
In particular $\langle P(...,a,0)\rangle 
= (A^{a}\delta +[a])\langle P(...,0)\rangle$.
\end{lemma}

\begin{remark} \label{remark}
$\langle P(a)\rangle = \delta (\!-A^{-3})^{a}$ since $P(a)$ is the trivial knot diagram with $|a|$ kinks, negative kinks if $a>0$, positive kinks if $a<0$.
\end{remark}

\section{Span of the Jones polynomial of pretzel links} \label{SectionSpanJones}
If $L$ is an oriented link, we denote by $V(L)$ its Jones polynomial with normalization $V(\put(0,3){\unknot } \hspace{0.25cm})
=-t^{-1/2}-t^{1/2}$ (see \cite{Lickorish}). Recall that $V(L)=(-A)^{-3w(D)}\langle D\rangle $ after the substitution
$A=t^{-1/4}$, where $D$ is an oriented diagram of $L$ and $w(D)$ is its writhe. It follows that $\tn{span}(\langle D\rangle
) = 4 \, \tn{span}(V(L))$.

\begin{remark}
Denote by $\langle D\rangle _{1}$ the Kauffman bracket of the diagram $D$ defined through the same relations (i)
and (ii) but with normalization $\langle \put(0,3){\unknot } \hspace{0.25cm} \rangle _{1}=1$. It follows that $\langle D\rangle =\delta \langle D\rangle _{1}$ for every link diagram $D$.
In parallel, denote by $V_1(L)$ the Jones polynomial of the oriented link $L$ with normalization $V_1(\put(0,3){\unknot } \hspace{0.25cm})=1$. Recall that $V_1(L) = (-A)^{-3w(D)}\langle D\rangle _1$ after the
substitution $A=t^{-1/4}$, where $D$ is an oriented diagram of $L$ and $w(D)$ is its writhe. It follows that $V(L) = (-t^{-1/2}-t^{1/2})V_{1}(L)$ for every oriented link $L$, and hence $\tn{span}(V(L)) = 1 + \tn{span}(V_1(L))$.
\end{remark}

\begin{notation} \label{parametros} For a pretzel link diagram $P(a_1, \ldots ,a_n)$, we write $r$ for the number of $a_i>1$, $s$ for the number of $a_i<-1$, $z$ for the number of $a_i=0$, $\alpha $ for the number of $a_i=1$ and $\beta $ for the number of $a_i=-1$. We also set $\lambda =\alpha -\beta$. Finally, let $\Sigma = \Sigma (P(a_1, \ldots, a_n)) = \sum_{|a_i|>1} |a_i|$.
\end{notation}

\begin{remark} \label{obs} Recall that:
\begin{enumerate}
\item \label{obs1} For any permutation $\sigma $ of $\{1,\ldots ,n\}$ the pretzel link defined by the  diagram $P(a_{\sigma (1)},\dots
,a_{\sigma (n)})$ can be obtained applying a finite sequence of mutations to the link defined by 
$P(a_{1},\dots ,a_{n})$, and thus their Jones polynomials agree. That the Kauffman bracket of the corresponding diagrams agree follows directly from the formula in Theorem~\ref{TheoremMain}.

\item \label{obs2}  
If $\bar{L}$ is the mirror image of the oriented link $L$, then $V(\bar{L})$ is obtained from $L$ by interchanging $t$ and $t^{-1}$, so both polynomials share span. In particular the span of the Jones polynomials of the pretzel links defined by the diagrams $P(a_1, \ldots, a_n)$ and $P(-a_1, \ldots, -a_n)$ is the same. Notice also that the symmetry $a_i \mapsto -a_i$ interchanges $r$ and $s$ and takes $\lambda$ into~$-\lambda$.

\item \label{obs3}  Any consecutive pair $(1,-1)$ in $(a_1,\ldots ,a_n)$ can be canceled via a type~II Reidemeister move without changing the link type. Moreover, after a $\pi$-rotation, any pair $(1,-1)$ can be canceled, even if they are non-consecutive (see Figure \ref{FigureUnomenosuno}). If there are only entries $\pm 1$, and in the same amount, the link is just the split union of two trivial knots, ${\put(0,3)\unknot \, \put(9,3) \unknot \hspace{0.55cm}}$, with Jones polynomial $(-t^{-1/2}-t^{1/2})^2$ and span $2$.
\begin{figure}[ht!]
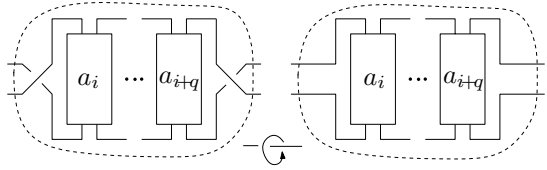
 
\labellist
 \pinlabel {$a_i$} at 87 90
 \pinlabel {$a_{i\!+\!q}$} at 185 90
 \pinlabel {$a_i$} at 395 90
 \pinlabel {$a_{i\!+\!q}$} at 491 90
\endlabellist
\begin{center}
\unomenosuno
\end{center}
\caption{Canceling a pair $(1,-1)$, consecutive or not}
\label{FigureUnomenosuno}
\end{figure}

\item \label{obs4} 
Similarly, if some entry $a_i$ is equal to 0, then any entry $-1$ or $+1$ can be deleted by a $\pi$-rotation without changing the link type. 

\item \label{obs5} For $n\geq 2$, the pretzel diagrams $P(1,-2,a_3,\dots, a_n)$ and $P(2, a_3,\dots, a_n)$ define the same pretzel link. Indeed, this can be achieved by a combination of a type~III plus a type~I Reidemeister moves (see Figure \ref{FigureTresuno}). 
\begin{figure}[ht!]
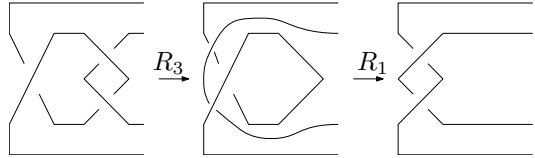
 
\labellist
 \pinlabel {$R_3$} at 170 98
 \pinlabel {$R_1$} at 390 98
\endlabellist
\begin{center}
\tresuno
\end{center}
\caption{$P(1,-2, a_3,  \dots, a_n)$ is isotopic to $P(2, a_3, \dots, a_n)$}
\label{FigureTresuno}
\end{figure}
\end{enumerate}
\end{remark}
 
Based on Remarks \ref{obs}.\ref{obs3} and \ref{obs}.\ref{obs4} we introduce the following definition:

\begin{definition}
We say that a pretzel diagram $P$ is {\it reduced} if it is $P=P(1,-1)$ or satisfies the following two conditions:
\begin{itemize}
\item[(i)] $\alpha\beta=0$, i.e., it does not have simultaneously entries equal to $+1$ and $-1$, and 
 \item[(ii)] if $z\not=0$, then $\alpha=\beta=0$.
\end{itemize}   
\end{definition}

\begin{theorem} \label{TheoremSpan} 
Let $P(a_1, \ldots ,a_n)$ be an unoriented, reduced pretzel diagram of an oriented pretzel link $L$, with $r$, $s$, $z$, $\lambda$ and $\Sigma$ as in Notation~\ref{parametros}. Let $S$ be the span of the Jones polynomial $V(L)$ of $L$. 
Since we are interested in the calculus of this span, we can assume that $a_1\geq \dots \geq a_n$ by Remark \ref{obs}.\ref{obs1}. Then:
\begin{enumerate}
\item $S = \Sigma + z$ if $z>0$.
\end{enumerate}
In the remaining cases we assume $z = 0$.
\begin{enumerate}
\setcounter{enumi}{1}
\item $S = \Sigma - \min \{1,r + \lambda ,s - \lambda \} + 1$ if $r + \lambda \neq 1$ and $s - \lambda \neq 1$, except the case $P(1,-1)$ for which $S=2$.
\end{enumerate}
In the remaining cases we consider $z=0$ and $r+\lambda =1$ (taking the mirror image the case $s-\lambda=1$ can be reduced to this one by Remark  \ref{obs}.\ref{obs2}).
\begin{enumerate}
\setcounter{enumi}{2}
\item In this item we assume $r>1$.
\begin{enumerate}
	\item $S = \Sigma - 1$ if $(r,\lambda, s) \not= (2,-1,0)$.
\end{enumerate}
In the rest of the cases in this item we assume that $(r,\lambda, s) = (2,-1,0)$.	
\begin{enumerate}
\setcounter{enumii}{1}
	\item $S= \Sigma -2 = 2$  if $a_2=2$ and $a_1=2$.
	\item $S= \Sigma -4 = 1$  if $a_2=2$ and $a_1=3$.
	\item $S= \Sigma -3 = a_1-1$  if $a_2=2$ and $a_1>3$.
	\item $S= \Sigma -2 = a_1+a_2-2$  if $a_2 > 2$.
	\end{enumerate}
\item In this item we assume $r=1$ and $s\leq 1$ (so $a_1>1$ and, if $s=1$, $-1>a_2$).
 \begin{enumerate} 
	\item $S = 1 = \Sigma - a_1 +1$ if $s=0$. 
	\item $S = 2$ if $s=1$ and $|a_1+a_2| = 0$.
	\item $S = 1$ if $s=1$ and $|a_1+a_2| = 1$.
	\item $S = 1+ |a_1+a_2|$ if $s=1$ and $|a_1+a_2| > 1$.
\end{enumerate}
\item In this item we assume $r=1$ and $s > 1$ (so $a_1>1$ and $-1 > a_2\geq a_3\geq \dots$).
 \begin{enumerate} 
	\item $S = \Sigma - \min \{a_{1},|a_{2}| - 1\}$ if $a_1\not = |a_2|-1$.
	\item $S = \Sigma - \min \{|a_{2}|,|a_{3}| - 1\}$ if $a_1 = |a_2|-1$ and $|a_{2}| \neq |a_{3}| - 1$. 
	\item $S=2a_1$ if $s=2$, $a_1 = |a_2|-1$ and $|a_{2}| = |a_{3}| - 1$, except the case $P(2,-3,-4)$, for which $S = 3$. 
	\item $S = \Sigma - a_1 - 3$ if $a_1 = |a_2|-1$, $|a_2| = |a_3| - 1$ and $|a_3| < |a_4| -1$, except the cases $P(2,-3,-4,a_4)$ with $a_4<-6$, for which~$S = \Sigma - 6$.
	\item $S = \Sigma - a_1 - 2$ if $a_1 = |a_2|-1$, $|a_2| = |a_3| - 1$ and $|a_3| = |a_4| -1$.
      \item $S = \Sigma - a_1 - 1$ if $a_1 = |a_2| - 1$, $|a_2| = |a_3| - 1$ and $|a_3| = |a_4|$.
 \end{enumerate} 
\item In this item we assume $r=0$ and $a_2 \not= -2$ (so $a_1=1$, $a_2<-2$ or $s=0$).
\begin{enumerate}
	\item $S = \Sigma +1 = 1$ if $s = 0$. 
	\item $S = \Sigma = -a_2$ if $s = 1$. 
	\item $S = \Sigma - 2$ if $s = 2$. 
	\item $S = \Sigma - 1$ if $s>2$.
\end{enumerate}
\item In this item we assume $r=0$ and $a_2 = -2$ (so $a_1=1$ and $a_2=-2$). 
\begin{enumerate}
	\item $S= \Sigma -1 = 1$ if $s=1$.
	\item $S= \Sigma -2 = 2$ if $s=2$ and $a_3=-2$.
	\item $S= \Sigma -4 = 1$ if $s=2$ and $a_3=-3$.
	\item $S= \Sigma -3 = -1-a_3$ if $s=2$ and $a_3<-3$.
	\item $S= \Sigma -1 = 3-a_4$ if $s=3$ and $a_3=-2$.
	\item $S= \Sigma -2 = 6$ if $s=3$, $a_3=-3$ and $a_4=-3$.
	\item $S= \Sigma -6 = 3$ if $s=3$, $a_3=-3$ and $a_4=-4$.
	\item $S= \Sigma -3 = 2-a_4$ if $s=3$, $a_3=-3$ and $a_4<-4$.
	\item $S= \Sigma -2 = -a_3-a_4$ if $s=3$ and $a_3<-3$.
	\item $S=\Sigma -1$ if $s>3$ and $a_3=-2$.
	\item $S=\Sigma -2$ if $s>3$, $a_3=-3$ and $a_4=-3$.
	\item $S=\Sigma -3$ if $s>3$, $a_3=-3$, $a_4=-4$ and $a_5=-4$.
	\item $S=\Sigma -4$ if $s>3$, $a_3=-3$, $a_4=-4$ and $a_5=-5$.
	\item $S=\Sigma -5$ if $s>3$, $a_3=-3$, $a_4=-4$ and $a_5<-5$, except for the pretzel $P(1, -2, -3, -4, a_5)$ with $a_5<-6$, for which $S=\Sigma - 6$.
	\item $S=\Sigma -3$ if $s>3$, $a_3=-3$ and $a_4<-4$.
	\item $S=\Sigma -2$ if $s>3$ and $a_3<-3$.
	
\end{enumerate}
\end{enumerate}
\end{theorem}

\begin{remark}
The exception stated in item {\it 5.4} of the previous theorem amendments the original statement of Theorem 2 in \cite {Pedro}. The proof of this theorem, including this correction and more details of item {\it 5} is given in a final appendix. 
\end{remark}

\begin{remark}
Item {\it 1} of Theorem \ref{TheoremSpan} analyzes when there is an entry equal to zero. The rest of cases assume that there are no entries equal to zero. Item~{\it 2} considers the generic case: $(r, \lambda, s)$ is neither in the plane $r+\lambda =1$ nor in $s-\lambda =1$. The symmetry $a_i \mapsto -a_i$ interchanges $r$ and $s$ and takes $\lambda$ into $-\lambda$, so $(r, \lambda, s) \mapsto (s, -\lambda, r)$. It follows that it is enough to give a formula for the span assuming that the point $(r, \lambda, s)$ is in the plane $r+\lambda =1$ (light blue plane in Figure \ref{FigurePolitopo}). We then distinguish cases $r>1$, $r=1$ and $r=0$. We deal with case $r>1$, except for the point $(r, \lambda, s)= (2,-1, 0)$, in {\it 3.1}. Note that $r+\lambda =1$, $r>1$ and $s-\lambda =1$ implies $(r, \lambda, s) = (2,-1,0)$. By symmetry this case can be studied as the case $(r,\lambda,s)=(0,1,2)$. We deal with case $r=1$ and $s\leq 1$ in {\it 4}. The difficult case, when $r=1$ and $s>1$, is treated in {\it 5}. Finally items {\it 6} and {\it 7} consider the cases when $r=0$, depending on whether or not $a_2=-2$ (see Figure \ref{FigurePlano}). Note that item {\it 7} can be directly deduced from items {\it 4} and {\it 5} by Remark~\ref{obs}.\ref{obs5}, but it has been included in the statement for practical computation purposes.
\end{remark}

\begin{figure}[ht!]
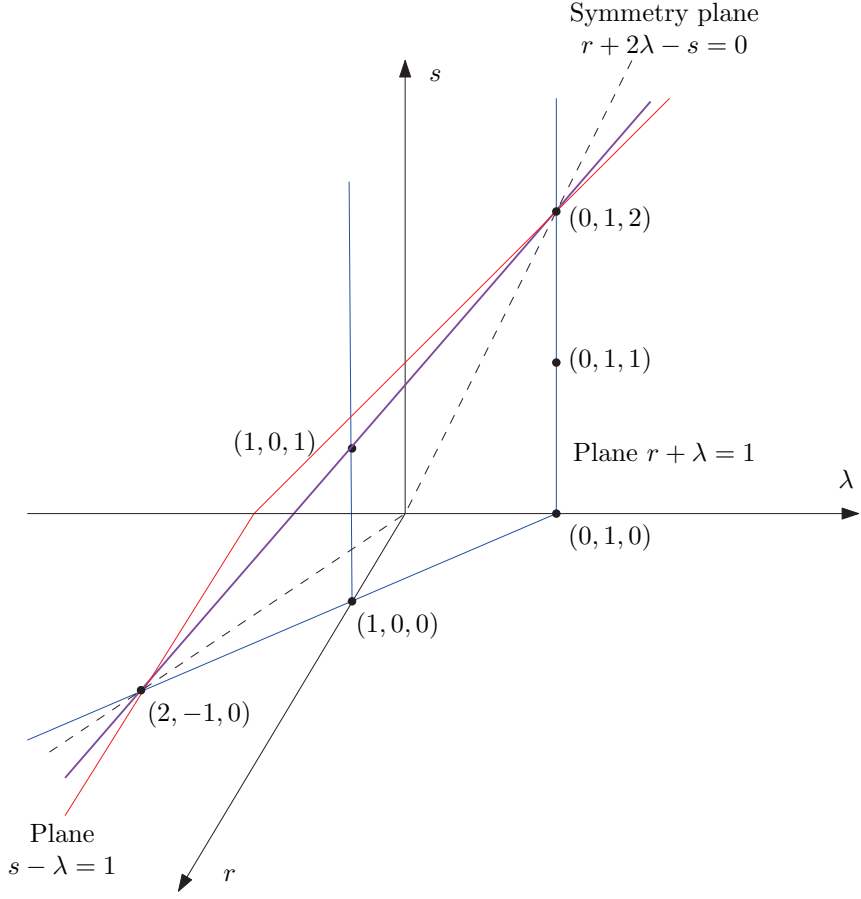
 
\labellist
 \pinlabel {\begin{tabular}{c} Plane $r+\lambda =1$ \end{tabular}} at 270 185
 \pinlabel {\begin{tabular}{c} Plane \\ $s - \lambda = 1$ \end{tabular}} at 15 17
 \pinlabel {\begin{tabular}{c} Symmetry plane \\  $r + 2\lambda - s = 0$ \end{tabular}} at 270 365 
 \pinlabel {$r$} at 86 7
 \pinlabel {$\lambda$} at 347 175
 \pinlabel {$s$} at 173 346
 \pinlabel {$(0,1,0)$} at 247 150
 \pinlabel {$(0,1,1)$} at 247 225
 \pinlabel {$(0,1,2)$} at 247 285
 \pinlabel {$(1,0,0)$} at 157 113
 \pinlabel {$(1,0,1)$} at 105 190
 \pinlabel {$(2,-1,0)$} at 73 75
\endlabellist
\begin{center}
\politopo
\end{center}
\caption{Cases in Theorem \ref{TheoremSpan} according to the values of $(r, \lambda, s)$}
\label{FigurePolitopo}
\end{figure}

\begin{figure}[ht!] 
\labellist
\pinlabel {$r$} at 543 33
\pinlabel {$s$} at 97 480
\pinlabel{\small \begin{tabular}{cc} Item {\it 6.4} \\ $a_2\not= -2$ \end{tabular}} at 0 445 
\pinlabel{\small \begin{tabular}{cc} Expand and \\ calculate \end{tabular}} at 0 368 
\pinlabel{\small \begin{tabular}{cc} Item {\it 6.3} \\ $a_2\not= -2$ \end{tabular}} at 0 275 
\pinlabel{\small \begin{tabular}{cc} Item {\it 7.1} \\ $a_2=-2$ \end{tabular}} at 0 150 
\pinlabel{\small \begin{tabular}{cc} Items {\it 7.5-16} \\ $a_2 = -2$ \end{tabular}} at 165 440 
\pinlabel{\tiny \begin{tabular}{cc} Items {\it 7.2-4} \\ $a_2 = -2$ \end{tabular}} at 205 220 
\pinlabel{\small Items {\it 4.2-4}} at 275 160 
\pinlabel{\small \begin{tabular}{cc} Item {\it 6.2} \\ $a_2 \not= -2$ \end{tabular}} at 135 175 
\pinlabel{\small \begin{tabular}{cc} Item {\it 1} \\ $z>0$ \end{tabular}} at 160 100 
\pinlabel{\small Item {\it 5}} at 290 368 
\pinlabel{\small Item {\it 3.1}} at 420 280
\pinlabel{\small Item {\it 3.2}} at 395 55
\pinlabel{\small Item {\it 6.1}} at 135 10 
\pinlabel{\small Item {\it 4.1}} at 265 10 
\pinlabel{\small $s-\lambda = 1$} at 420 10 
\endlabellist
\begin{center}
\plano
\end{center}
\caption{Plane $r+\lambda =1$ coordinated by $(r, s)$. An oval means that the corresponding knot is the unknot}
\label{FigurePlano}
\end{figure}

\section{There are finitely many pretzel links with a fixed span of the Jones polynomial} \label{SectionFewCases}

Suppose that $D$ is a connected diagram with $n$ crossings of an oriented link $L$ with Jones polynomial $V(L)$. Recall that $\tn{span}(V(L)) \leq n+1$ (see \cite{Lickorish}).

\begin{theorem} \label{TheoremAcotando}
Let $P(a_1, \dots, a_n)$ be a pretzel diagram of a pretzel link~$L$. Let $V(L)$ be the Jones polynomial of $L$, where $L$ has been arbitrarily oriented. Suppose that $L$ is not a torus link $T(2,n)$ with two strands. Then 
$$
\tn{span }(V(L))
\geq 
\sum_{|a_i|>1} |a_i| - \min \{ |a_i| \, /\, |a_i|>1\} - 4.
$$
\end{theorem}

\begin{proof}
Let $\Sigma = \sum_{|a_i|>1} |a_i|$ and $M = \min \{ |a_i| /  |a_i|>1\}$ if $r+s>0$. If $r=s=0$ let $\Sigma = M = 0$. Let $S= \tn{span} (V(L))$. We want to prove that $S \geq \Sigma - M - 4$. 

Clearly we may assume that the pretzel diagram is reduced (note that if $P=P(1,-1)$, then  $S=2$ and $\Sigma -M - 4 = -4$). We may also assume that $a_1\geq a_2 \geq \dots \geq a_n$ by Remark~\ref{obs}.\ref{obs1}, hence the pretzel diagram $P(a_1, \dots, a_n)$ is under the hypothesis of Theorem~\ref{TheoremSpan}.

If $z>0$ then $S = \Sigma + z > \Sigma > \Sigma - M - 4$ by Theorem \ref{TheoremSpan}.1, since $M\geq 0$. Hence we may assume that $a_i \not =0$ for any $i=1, \dots , n$.

Assume that $r+\lambda \not=1$ and $s-\lambda \not= 1$. Then by Theorem \ref{TheoremSpan}.2
$$
S = \Sigma - \min \{ 1, r+\lambda, s-\lambda\} +1 
\geq \Sigma - 1 +1 
= \Sigma > \Sigma - M - 4.
$$ 

Suppose that $s-\lambda =1$. Consider then the pretzel diagram $P'=P(-a_1, \dots, -a_n)$. On one hand, the values for $\Sigma$, $M$ and the bound $\Sigma-M-4$ are unchanged, since $\Sigma$ and $M$ are defined in terms of absolute values. On the other hand, the span of the new pretzel link is the same, by Remark \ref{obs}.\ref{obs2}. And for this new pretzel diagram $P'$ we have that $r'+\lambda' = s + (-\lambda) =1$. In other words, in order to conclude the proof we may assume that $r+\lambda =1$ (and $z=0$).

$\bullet$ Suppose that $r>1$. Then, by Theorem \ref{TheoremSpan}.3 we have $S \geq \Sigma -4 \geq \Sigma - M - 4$.

$\bullet$ Suppose that $r=1$, hence $\lambda = 0$. Recall that we have already deleted the pairs $+1, -1$, and in addition we have $a_1 \geq 2$, $a_i \leq -2$ for $i=2, \dots, n$ and $|a_2| \leq |a_3| \leq \dots |a_n|$.

If $s=0$ then $S = 1 = \Sigma - a_1 + 1 = \Sigma -M +1 > \Sigma - M - 4$ by Theorem~\ref{TheoremSpan}.4.1.

If $s>1$, we apply Theorem~\ref{TheoremSpan}.5:
\begin{enumerate}
\item If $a_{1}\neq |a_{2}| - 1$ then by Theorem~\ref{TheoremSpan}.5.1
$$
S = \Sigma - \min \{a_{1},|a_{2}| - 1\} 
\geq \Sigma - \min \{a_{1},|a_{2}|\}
= \Sigma - M 
> \Sigma - M - 4.
$$

In the following cases $a_1=|a_2|-1$ hence $M= a_1$. We have:
\item If $a_{1} = |a_{2}| - 1$ and $|a_{2}| \neq |a_{3}|-1$ then by Theorem~\ref{TheoremSpan}.5.2
$$
S = \Sigma - \min \{|a_{2}|,|a_{3}| - 1\} 
\geq \Sigma - (M + 1) > \Sigma - M - 4.
$$
\item If $s=2$,  $a_{1} = |a_{2}| - 1$ and $|a_{2}| = |a_{3}| - 1$ but the diagram is not $P(2,-3,-4)$, then
$$
\Sigma - M -4 
= (a_1-a_2-a_3)-a_1-4
= -a_2-a_3-4
= 2a_1-1 
< 2a_1
= S
$$
and for $P(2,-3,-4)$ the bound is sharp:  $\Sigma -M -4 = 9-2-4=3 = S$.

\item If $a_{1} = |a_{2}| - 1$, $|a_{2}| = |a_{3}| - 1$ and $|a_{3}| < |a_{4}| -1$ then by Theorem~\ref{TheoremSpan}.5.4
$$
S = \Sigma - a_1 - 3 > \Sigma - M - 4
$$
excepting the cases $P(2,-3,-4,a_4)$ with $a_4<-6$, where $S = \Sigma - 6 = \Sigma - 2 - 4 = \Sigma - M - 4$, again a case for which the bound is sharp.
\item If $a_{1} = |a_{2}| - 1$, $|a_{2}| = |a_{3}| -1$ and $|a_{3}| = |a_{4}| -1$ then by Theorem~\ref{TheoremSpan}.5.5
$$
S = \Sigma - a_1 - 2 = \Sigma - M - 2 > \Sigma - M - 4.
$$
\item If $a_{1} = |a_{2}| - 1$, $|a_{2}| = |a_{3}| - 1$ and $|a_{3}| = |a_{4}|$ then by Theorem~\ref{TheoremSpan}.5.6
$$
S = \Sigma - a_1 - 1 = \Sigma - M - 1 > \Sigma - M - 4.
$$
\end{enumerate}

We do not have to consider the case $s=1$ since $P(a, b) \sim P(0, a + b) = T(2, a + b)$ is a torus link with two strands. 

$\bullet$ Finally suppose that $r = 0$. Then the bound can be directly checked by simple inspection of items {\it 6} and {\it 7} of Theorem~\ref{TheoremSpan}, having in mind that $M=2$ for item~{\it 7} ({\it 7.7} and the exception in {\it 7.14} are cases of sharp bound).
\end{proof}

\begin{remark}\label{torus}
We must discard torus links $T(2,n)$ with two strands by the following: if $5 < a <- b-1$ in $P = P(a, b) \sim P(0, a + b) = T(2, a + b)$, then on one hand the Jones polynomial of the link has span $S = 1 + |a+b| = 1 - a - b$ by Theorem \ref{TheoremSpan}.1 and on the other hand $\Sigma - M - 4 = (a - b) - a - 4 = - b - 4$. And $1-a-b < -b - 4$ since $5<a$. This shows in particular that there are infinitely many pretzel diagrams with the same span. In the example all these diagrams correspond to the same torus link with two strands. Next we will see that these are the only exceptions. 
\end{remark}

\begin{theorem} \label{TheoremFewCases}
Given a natural number $S$, there are finitely many oriented pretzel links whose Jones polynomials have span $S$.
\end{theorem}
\begin{proof}
This follows from Theorem \ref{TheoremAcotando}, items {\it 1} and {\it 2} in Theorem \ref{TheoremSpan} and the fact that the span of $V(L)$ is $|q|+1$ if $L$ is the torus link $T(2,q)$ (see Remark \ref{torus}). Let us see the details of the proof. Let $L$ be an oriented pretzel link, represented by the unoriented   pretzel diagram $P = P(a_1, \ldots, a_n)$, with span$(V(L)) = S$. By Remarks \ref{obs}.3 and \ref{obs}.\ref{obs4}, we may choose this diagram to be reduced. And we may assume that $P$ is under the hypothesis of Theorem~\ref{TheoremSpan}, having in mind that there are $n!$ possible reorderings that could give different pretzel links, although all of them would have the same Jones polynomial, hence the same span.

If $L$ is a torus link with two strands, it is necessarily $T(2, S-1)$ or its mirror image, only two possibilities. Next suppose that $L$ is not a torus link with two strands. Then by Theorem \ref{TheoremAcotando}, and with its notation, 
$$
S = \tn{span}(V(L)) 
\geq \Sigma - M - 4 
\geq 2(r+s-1)-4 
= 2(r+s-3)
$$
from which we deduce that $r+s \leq \frac{S}{2} +3$. Also, by Theorem \ref{TheoremSpan}.1, $z \leq S$.

Taking into account that $r+s$ has an upper bound, the possible value of $\lambda = \alpha - \beta$ has also lower and upper bounds by Theorem \ref{TheoremSpan}.2, assumed that $z=0$ (in Proposition~\ref{PropositionBounds} we will determine specific bounds for $\lambda$). This proves that the number $n$ of entries has necessarily an upper bound as long as $z=0$. 

Assume now that $r+s>1$ (and $z=0$). Then any entry $a_i$ satisfy that $|a_i| \leq S+4$ since, if for example $M = |a_k|$ and $|a_j|>1$, $j\not= k$,~then 
$$
S 
\geq \Sigma - M - 4  
= \sum_{|a_i|>1} |a_i| - |a_k| - 4
= \sum_{|a_i|>1, i\not=k} |a_i| - 4
\geq |a_j| - 4.
$$
And, indeed, $|a_k| \leq |a_j| \leq S+4$.
This leaves a finite number of possibilities for the numbers $a_i$ with $|a_i|>1$ and proves the statement if $r+s>1$.

Consider now the case $r+s =1$ (and $z=0$). By symmetry we may assume $r=1$, $s=0$. If $\lambda \not = -1, 0$, then $r+\lambda \not= 1$ and $s-\lambda \not= 1$, and by Theorem~\ref{TheoremSpan}.2 $S = |a|-\min \{1, r+\lambda, s-\lambda\} +1 \geq |a|$, hence $|a|< S$ where $a$ is the only entry with $|a|>1$. So, a finite number of possibilities. If $\lambda  =-1$, $P=P(a, -1)$ is a torus link. If $\lambda =0$, then $P=P(a)$ is the unknot. 

Finally assume $z>0$. Since $P$ is reduced, $\alpha = \beta =0$. Then by Theorem~\ref{TheoremSpan}.1 we have that $n\leq S$ since 
$$
S = \Sigma +z 
\geq 2(n-z) + z = 2n-z \geq 2n - n = n
$$  
and $|a_i|<S$, since $|a_i| \leq \Sigma < S$. This completes the proof. 
 \end{proof}

We now collect specific upper bounds in the following result. As usual, we consider a reduced pretzel diagram $P(a_1, \dots, a_n)$ with $n$ entries representing a pretzel link whose Jones polynomial (with normalization as in Section \ref{SectionSpanJones}) has span $S$, and let $z$, $r$, $s$ and $\lambda$ be integers as in Notation~\ref{parametros}. 

\begin{proposition} \label{PropositionBounds}
We have that 
$$
z\leq S, \quad r+s\leq \frac{S}{2}+3
\quad \tn{ and } \quad 
|\lambda| \leq \max \{ \frac{S}{2} + 2, S-1\}.
$$ 
Moreover, $|a_i|\leq S+4$ for $i=1, \dots, n$, and 
$$
n \leq \max \{ 2S+5, \frac{5}{2}S+2\}.
$$
\end{proposition}
\begin{proof}
The bounds for $z$, $r+s$ and $|a_i|$ were obtained in the proof of Theorem~\ref{TheoremFewCases}. We now prove that $|\lambda| \leq \max \{ \frac{S}{2} + 2, S-1\}$ (note that the maximum is $S-1$ if $S>5$). If $\lambda=-1, 0$ or $1$ this is obvious. Now assume that $\lambda \geq 2$ (a symmetric argument works if $\lambda \leq -2$). First note that if  $\lambda > \frac{S}{2}+2$ then $s - \lambda < 1$, since
$$
\lambda > \frac{S}{2}+2 = \frac{S}{2}+3-1 \geq r+s-1 \geq s-1.
$$
Moreover, $s-\lambda \not= 1$ and $r+\lambda \not=1$ (since $\lambda \geq 2$), hence by Theorem 
\ref{TheoremSpan}.2 we have $S 
= \Sigma - \min\{ 1, r+\lambda, s-\lambda\}  +1
= \Sigma - (s-\lambda) + 1$ and, since $s-\Sigma \leq 0$, $\lambda = S -1 + (s- \Sigma) \leq S-1$.

We finally check the upper bound for the number $n$ of entries: 
$$
\begin{array}{rcl} 
n
& = & r + s + z + |\lambda| \\
& \leq & \frac{S}{2} + 3 + S + \max\{ \frac{S}{2}+2, S-1\} \, \leq \, \max\{ 2S+5,  \frac{5}{2}S+2 \}.
\end{array}
$$
\end{proof}

Looking closer at the proof of Theorem~\ref{TheoremFewCases}, we actually prove the following corollary.

\begin{corollary} \label{CorollaryFewCases}
\begin{itemize}
\item[a)]  Given a natural number $S$, there are finitely many reduced pretzel diagrams   and not representing a torus link with two strands, such that the Jones polynomials of the pretzel links represented by these diagrams have span $S$.

\item[b)] A pretzel link $L$ has a finite number of reduced pretzel diagrams if and only if $L$ is not a torus link with two strands.
\end{itemize}
\end{corollary}

\section{Pretzel knots up to nine  crossings}\label{SectionList}

In this section we provide the complete list of pretzel knots up to nine crossings. By abuse of language we will say that a link $L$ has span $S$ if its Jones polynomial $V(L)$, normalized as in Section \ref{SectionSpanJones}, has span $S$ (note that the span is not affected by the orientation of the link); also, a pretzel diagram $P=P(a_1, \dots, a_n)$ has span $S$ if the  link $L$ defined by $P$ has span $S$. In the same line we will speak about $V(L)$ as simply the Jones polynomial of the pretzel diagram $P$. 

\begin{theorem}\label{TheoremUpToEightCrossings}
The following table gives the information about pretzel and non-pretzel knots up to 8 crossings. For the pretzel ones, we provide a pretzel diagram. In particular, the first non-pretzel knot in tables is the knot $8_{12}$.
\end{theorem}
$$
\begin{array}{l}
3_1 = P(1,1,1)\\ 
4_1=P(1,1,2)\\ 
5_1 =P(1,1,1,1,1)\\ 
5_2 =P(1,1,3)\\
6_1 =P(1,1,4)\\
6_2 =P(1,2,3)\\
6_3 =P(2,1,-3,1)\\
7_1 =P(1,1,1,1,1,1,1)\\
7_2 =P(5,1,1)\\
7_3 =P(1,1,1,4)\\
7_4 =P(3,1,3)\\
7_5 =P(2,1,1,3)\\
7_6 =P(2,1,1,-3,1)\\
7_7 =P(1,1,1,-3,-3)\\   
\end{array}
\qquad
\begin{array}{l}
8_1 = P(1,1,6)\\
8_2 = P(1,2,5)\\
8_3 = P(1,1,1,1,4)\\
8_4 = P(1,3,4)\\
8_5 = P(2,3,3)\\
8_6 = P(1,1,1,2,3)\\
8_7 = P(4,1,-3,1)\\   
8_8 = P(2,1,1,1,-3,1)\\  
8_9 = P(3,1,-4,1)\\  
8_{10} = P(-3,-2,3,-1)\\  
8_{11} = P(3,1,1,-3,1)\\    
8_{12} \tn{ is not pretzel} \\   
8_{13} = P(1,1,1,-3,-4)\\  
8_{14} \tn{ is not pretzel} \\
8_{15} = P(-3,1,2,1,-3) \\   
8_{16} \tn{ is not pretzel} \\
8_{17} \tn{ is not pretzel} \\
8_{18} \tn{ is not pretzel} \\
8_{19} = P(3,3,-2)\\   
8_{20} = P(-3,2,3,-1)\\    
8_{21} = P(-3,-3,1,2) 
\end{array}
$$
Some of the above knots were known to be pretzel (see for example tables in \cite{Peter} or \cite{KnotAtlas}), and other have been recently matched as pretzel in \cite{Clara}. That the above list is now exhaustive can be deduced  exactly in the same way as in the proof of Theorem~\ref{TheoremNineCrossings} below for knots with nine crossings.

As an example, the Jones polynomial of the non-alternating knot with eight crossings $8_{21}$ is $V_1(8_{21})=	2t - 2t^2 + 3t^3 -3t^4 + 2t^5 -2t^6 + t^7$ hence $\tn{span}(V(8_{21})) = 6+1 = 7$. Then we list all the pretzel links whose Jones polynomial have span equal to $7$ and observe that the Jones polynomial of $P(3,3,-1,-2)$ coincides with that of $8_{21}$. Although there is no such coincidences for small knots, we check by hand that both knots are the same (see Figure~\ref{OVU}). 
\begin{figure}[ht!]
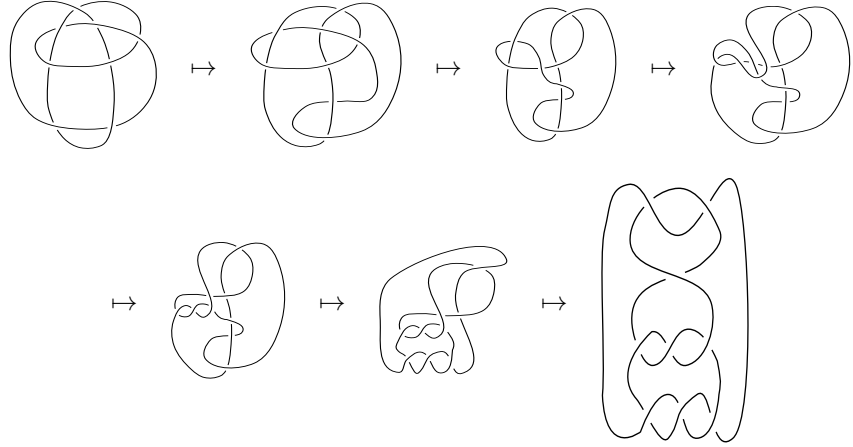
  
\begin{center}
$$
\OVUpretzelUno 
\quad \mapsto \quad
\OVUpretzelDos
\quad \mapsto \quad
\OVUpretzelTres
\quad \mapsto \quad
\OVUpretzelCuatro
$$
$$
\mapsto \quad
\OVUpretzelCinco
\quad \mapsto \quad
\OVUpretzelSeis
\quad \mapsto \quad
\OVUpretzelSiete
$$
\end{center}
\caption{The knot $8_{21}$ is the pretzel knot $P(2, 1, -3, -3)$.}
\label{OVU} 
\end{figure}

\begin{theorem}\label{TheoremNineCrossings}
The following table gives the information about pretzel and non-pretzel knots with 9 crossings. For the pretzel ones, we provide a pretzel diagram. 
In particular the first non-alternating and non-pretzel knot is the knot $9_{42}$. In fact, all the non-alternating knots with nine crossings are non-pretzel except $9_{46}$ and $9_{48}$.
\end{theorem}
$$
\begin{array}{l}
9_1=P(1,1,1,1,1,1,1,1,1) \\ 
9_2=P(1,1,7) \\ 
9_3=P(1,-4,5)  \\ 
9_4=P(1,-5,4) \\ 
9_5=P(1,3,5) \\  
9_6=P(1,1,-3,-6) \\ 
9_7=P(-3,-1,-1,-1,-1,-2) \\
9_8=P(2,1,1,1,1,-3,1) \\
9_9=P(-4,-1,-1,-3) \\
9_{10}=P(-3,-1,-1,-1,-3)  \\ 
9_{11}=P(5,-2,-1,-1,-1) \\
9_{12}=P(4,1,1,-3,1) \\
9_{13}=P(3,1,1,-4,1) \\
9_{14}=P(1,1,1,-3,-5)  \\ 
9_{15} \hbox{  is not pretzel}\\
9_{16}=P(2,3,1,3) \\ 
9_{17}=P(-3,1,1,1,1,1,-3) \\ 
9_{18}\hbox{  is not pretzel} \\
9_{19} \hbox{  is not pretzel}\\
9_{20}=P(1,1,1,1,-3,-4) \\ 
9_{21} \hbox{  is not pretzel}\\
9_{22} \hbox{  is not pretzel} \\
9_{23} \hbox{  is not pretzel}\\
9_{24}=P(3,1,-3,1,2) \\ 
9_{25} \hbox{  is not pretzel}\\
\end{array}
\qquad
\begin{array}{l}
9_{26} \hbox{  is not pretzel}\\
9_{27} \hbox{  is not pretzel}\\
9_{28}=P(1,1,-3,2,1,-3) \\ 
9_{29} \hbox{  is not pretzel}\\
9_{30} \hbox{  is not pretzel}\\
9_{31} \hbox{  is not pretzel}\\
9_{32} \hbox{  is not pretzel}\\
9_{33} \hbox{  is not pretzel}\\
9_{34} \hbox{  is not pretzel}\\
9_{35}=P(3,3,3)\\ 
9_{36} \hbox{  is not pretzel}\\
9_{37}=P(1,-3,3,-3,1) \\ 
9_{38} \hbox{  is not pretzel}\\
9_{39} \hbox{  is not pretzel}\\
9_{40} \hbox{  is not pretzel}\\
9_{41} \hbox{  is not pretzel}\\
9_{42} \hbox{  is not pretzel} \\
9_{43} \hbox{  is not pretzel} \\
9_{44} \hbox{  is not pretzel} \\
9_{45} \hbox{  is not pretzel}\\
9_{46} =P(3,3,-3 ) \\ 
9_{47} \hbox{  is not pretzel}\\
9_{48} =P(-3,1,-3,1,-3) \\  
9_{49} \hbox{  is not pretzel}\\
\end{array}
$$

The following definition will be useful in the proof. Fix an integer $S$. A set~${\cal L}_S$ of pretzel  diagrams with span $S$ is said to be {\it complete} if any pretzel knot~$K$ with span~$S$ has a diagram which can be obtained from a diagram in ${\cal L}_S$ by mirror image and/or by reordering its entries. Results in Section~\ref{SectionFewCases} (precisely Theorem~\ref{TheoremFewCases} and Corollary~\ref{CorollaryFewCases}) guaranty that there exists a {\it finite} complete set~${\cal L}_S$ for each natural number~$S$. 
In Lemma \ref{LemmaSpanTen} we will find a finite complete set ${\cal L}_S$ for $S=10$, which is the key ingredient to prove Theorem \ref{TheoremNineCrossings}.

\begin{proof}
The Jones polynomial of the alternating knots $9_1$ to $9_{41}$ have indeed span $S=10$. Let ${\cal L}_{10}$ be the complete set of pretzel diagrams with span $10$ determined in Lemma~\ref{LemmaSpanTen} below. We compute their Jones polynomials by using Theorem~\ref{TheoremMain} and with the help of Sage.
Now, let $K$ be an alternating knot with 9 crossings. Let ${\cal L}_{10}(K)$ be the subset of $ {\cal L}_{10}$ consisting on those pretzel diagrams with Jones polynomial equal or symmetric to $V(K)$. If ${\cal L}_{10}(K)$ is empty we can confirm that $K$ is not pretzel. 
Otherwise $K$ could be still non-pretzel, but if it is pretzel, must have a diagram that can be obtained from a diagram $P=P(a_1, \dots, a_n) \in {\cal L}_{10}(K)$ by reordering its entries ($K$ would be a mutant of the knot with diagram $P$) and/or by taking its mirror image. In all the cases in which ${\cal L}_{10}(K)$ was non-empty we directly checked by hand, using Reidemeister moves, that $K$ was pretzel (since they are in general small knots with the same Jones polynomial, this was quite expected). 

For non-alternating knots $9_{42}$ to $9_{49}$ we have repeated the same strategy, but in these cases the considered set ${\cal L}_S$ corresponds to a span $S<10$ (for $S<10$ the set ${\cal L}_S$ is even smaller and can be found in the same way as in Lemma~\ref{LemmaSpanTen}). This completes the proof.  
\end{proof}

\begin{lemma} \label{LemmaSpanTen}
The following set ${\cal L}_{10}$ of reduced pretzel knot diagrams with span $S=10$ is complete (``Type" indicates the corresponding item of Theorem~\ref{TheoremSpan} satisfied; a pretzel diagram $P(a_1, \dots, a_n)$ will be shorted just by writing $(a_1, \dots, a_n)$):
\begin{itemize}
\item Type 1: 

$
\begin{array}{l}
(3,3,3,0), (3,3,0,-3), (9,0).
\end{array}
$

\item Type 2:

$
\begin{array}{l}
(7,2), (7,1,-2), (2,1,-7), (1,1,-2,-7),  \\
(3,3,3), (3,1,1,-3,-3),  \\
(6,3), (6,1,-3), (3,1,-6), (1,1,-3,-6),  \\
(5,4), (5,1,-4), (4,1,-5), (1,1,-4,-5),  \\
(3,3,2,1), (3,3,1,1,-2), (3,2,1,1,-3), \\ 
(3,1,1,1,-2,-3), (2,1,1,1,-3,-3), (1,1,1,1,-2,-3,-3), \\
(5,3,1), (1,1,1,-3,-5), \\
(8,1), (1,1,-8), \\
(5,2,1,1), (5,1,1,1,-2), (2,1,1,1,-5), (1,1,1,1,-2,-5), \\
(4,3,1,1), (4,1,1,1,-3), (3,1,1,1,-4), (1,1,1,1,-3,-4), \\
(7,1,1), \\
(3,3,1,1,1), (1,1,1,1,1,-3,-3), \\
(6,1,1,1), (1,1,1,1,-6), \\
(3,2,1,1,1,1), (3,1,1,1,1,1,-2), (2,1,1,1,1,1,-3), (1,1,1,1,1,1,-2,-3), \\
(5,1,1,1,1), \\
(4,1,1,1,1,1), (1,1,1,1,1,1,-4), \\
(3,1,1,1,1,1,1), \\
(2,1,1,1,1,1,1,1), (1,1,1,1,1,1,1,1,-2), \\
(1,1,1,1,1,1,1,1,1).
\end{array}
$

\item Type 3:

$
\begin{array}{l}
(3,3,3,2,-1,-1,-1), (3,3,3,-1,-1,-2), (3,3,2,-1,-1,-3),\\ 
(3,3,-1,-2,-3), (3,2,-1,-3,-3), \\
(5,3,3,-1,-1), \\
(11,2,-1), \\
(9,3,-1), \\
(7,5,-1).
\end{array}
$

\item Type 5:

$
\begin{array}{l}
(3, -3 ,-6), (6, -3,-3), \\
(3, -2, -3, -3), \\
(4, -3, -5), (5, -3 ,-4), \\
(2, -5, -5), (3, -5, -5), \\ 
(4,-5,-5), \\
(3, -4, -7), \\
(5, -6, -7).
\end{array}
$

\item Type 6: 

$
\begin{array}{l} 
(1,-3,-9), (1,-5,-7).
\end{array}
$
\end{itemize}
\end{lemma}

\begin{proof} We say that a pretzel diagram $P = P(a_1, \dots , a_n)$ is of type {\it i.j} if it satisfies the hypothesis of Theorem \ref{TheoremSpan}, item {\it i.j}. In particular, we will say that $P$ is generic if it is of type {\it 2}; otherwise is called non-generic. Recall that all these diagrams are reduced pretzel diagrams and in addition $a_1\geq a_2 \geq \dots \geq a_n$. We will run along all the items of Theorem \ref{TheoremSpan}. For each one, $i.j$, we will find the pretzel {\it knot} diagrams of type $i.j$ with span equal to $S = 10$, discarding a diagram if it can be obtained by mirror image or reordering  the entries of a previous selected one. Another important remark that will greatly reduce the list (and our effort) is that we will avoid to take pretzel diagrams with more than one even entry, since they correspond to links with two or more components. For the same reason, from the final list we will also delete the pretzel diagrams with $n$ even and all entries $a_i$ odd. We will use notation $r, s, z, \lambda$ and $\Sigma$ as in Theorem~\ref{TheoremSpan}. 
\begin{itemize}
\item If $P$ is of type 1 then $z>0$ and $S=\Sigma + z$. Since we want $P$ to represent a knot, necessarily $z=1$, hence $\Sigma =9$. On the other hand,  since $P$ is reduced, it is $\lambda=0$.
 Then the decompositions of $\Sigma =9$ as sum of positive numbers greater than 1 and none of them even ($0$ is already one even entry) are $3+3+3$ and $9$. 
We then need to assign different signs to the addends in order to obtain the different possibilities (in the rest of proof we will refer to this process as {\it assigning signs}), although avoiding to have diagrams that differ just by mirror image and reordering:
$$ 
(3, 3, 3, 0), (3, 3, 0, -3), (9, 0).
$$
\end{itemize}

Note that $z=0$ in the rest of items in Theorem~\ref{TheoremSpan}. Also, along the rest of proof, a {\it decomposition} of an integer $\Sigma \geq 2$ will mean a non-decreasing sequence of integers greater than one, at most one of them even, whose sum is exactly $\Sigma$. 

\begin{itemize}
\item Suppose that $P$ is of type 2, i.e., it is generic. 
Then $r+\lambda \neq 1$, $s-\lambda \neq 1$ and 
$S = \Sigma - \min \{1,r + \lambda ,s - \lambda \} + 1\geq \Sigma -1+1 =\Sigma$.
As $S=10$, the possible values of $\Sigma$ are $10,9,8,7,6,5,4,3,2$ and $0$. Let $m=\min \{1,r + \lambda ,s - \lambda \}$.
\begin{itemize}
\item If $\Sigma=10$ then $m=1$ and, since $r + \lambda \neq 1$ and $ s - \lambda \neq 1$, it must be $r+\lambda\geq 2$ and $s-\lambda\geq2$, thus $r+s\geq 4$. On the other hand, the maximum number of addends in a decomposition of $\Sigma=10$ occurs at $10=2+3+5$, which means that $r+s \leq 3$, and we conclude that no~$P$ satisfies this case. 
\end{itemize}

When $\Sigma\leq 9$, $m<1$ hence $m=r+\lambda$ or $m=s-\lambda$. One can also easily check that if $P$ is generic and $m=r+\lambda$, then its mirror image is also generic and satisfies that $m=s-\lambda$. Thus, for $\Sigma\leq 9$, we can assume that  $m=s-\lambda$ and therefore $r+\lambda\geq s-\lambda$.

\begin{itemize}
\item If $\Sigma=9$ then $m=s-\lambda=0$. Since $r+\lambda\not=1$ it follows that $r+s\not=1$ and the admissible decompositions of $\Sigma = 9$ are $2+7$, $3+3+3$, $3+6$ and $4+5$. Assigning signs to the addends and, since $\lambda=s$, adding as many 1's as negative addends, we obtain
$$
\begin{array}{l}
(7,2), (7,1,-2), (2,1,-7), (1,1,-2,-7) \\
(3,3,3), (3,3,1,-3), (3,1,1,-3,-3), (1,1,1,-3,-3,-3) \\
(6,3), (6,1,-3), (3,1,-6), (1,1,-3,-6),\\ 
(5,4), (5,1,-4), (4,1,-5), (1,1,-4,-5).
\end{array}
$$
\end{itemize}

For $0<\Sigma\leq 8$ we have $m=s-\lambda\leq -1$, hence $\lambda\geq s+1$, and the relations $r+\lambda\not=1,r+\lambda\geq -1$ are automatic. 

\begin{itemize}
\item If $\Sigma=8$, then $s-\lambda=-1$. Decompositions of $\Sigma =8$ are $2+3+3$, $3+5$ and $8$. Assigning signs and adding $\lambda=s+1$ entries equal to 1, we obtain 
$$
\begin{array}{l}
(3,3,2,1), (3,3,1,1,-2), (3,2,1,1,-3),\\
(3,1,1,1,-2,-3), (2,1,1,1,-3,-3), (1,1,1,1,-2,-3,-3),\\
(5,3,1), (5,1,1,-3), (3,1,1,-5), (1,1,1,-3,-5),\\
(8,1), (1,1,-8) .
\end{array}
$$

\item If $\Sigma=7$ then $s-\lambda=-2$. Decompositions of $7$ are $2+5$, $3+4$ and $7$. 
Assigning signs and adding $\lambda=s+2$ entries equal to 1, we obtain
$$
\begin{array}{l}
(5,2,1,1), (5,1,1,1,-2),(2,1,1,1,-5), (1,1,1,1,-2,-5),\\
(4,3,1,1), (4,1,1,1,-3), (3,1,1,1,-4), (1,1,1,1,-3,-4),\\
(7,1,1), (1,1,1,-7).
\end{array}
$$

\item If $\Sigma=6$ then $s-\lambda=-3$. Decompositions of $6$ are
$3+3$ and $6$. Assigning signs and adding $\lambda=s+3$ entries equal to 1, we obtain
$$
\begin{array}{l}
(3,3,1,1,1), (3,1,1,1,1,-3), (1,1,1,1,1,-3,-3),\\
(6,1,1,1), (1,1,1,1,-6).
\end{array}
$$

\item If $\Sigma=5$ then $s-\lambda=-4$. Decompositions of $5$ are $2+3$ and $5$.
Assigning signs and adding $\lambda=s+4$ entries equal to 1, we obtain
$$
\begin{array}{l}
(3,2,1,1,1,1), (3,1,1,1,1,1,-2), (2,1,1,1,1,1,-3), (1,1,1,1,1,1,-2,-3), \\
(5,1,1,1,1), (1,1,1,1,1,-5).
\end{array}
$$

\item If $\Sigma=4$ then $s-\lambda=-5$. The only decomposition of $4$ is $4$. Assigning signs and adding $\lambda=s+5$ entries equal to 1, we obtain
$$
\begin{array}{l}
(4,1,1,1,1,1), (1,1,1,1,1,1,-4).
\end{array}
$$

\item Similarly, for $\Sigma=3$ and $\Sigma=2$ we obtain the lists
$$
\begin{array}{l}
(3,1,1,1,1,1,1), (1,1,1,1,1,1,1,-3), \\
(2,1,1,1,1,1,1,1), (1,1,1,1,1,1,1,1,-2).
\end{array}
$$

\item Finally, for $\Sigma=0$ we have $s-\lambda=-9$. In this case $r=s=0$ and we obtain the pretzel diagram
$$
(1,1,1,1,1,1,1,1,1).
$$
\end{itemize}

\item If $P$ is of type 3.1 then $r+\lambda=1$, $r>1$ and $(r,\lambda,s)\not=(2,-1,0)$, thus $r+s\geq 3$ and $S=\Sigma-1$. The decompositions of $\Sigma=S+1=11$ with at least $3$ numbers are $2+3+3+3$ and $3+3+5$. Assigning signs taking into account that $r>1$ and adding $r-1$ entries equal to $-1$, we obtain
$$
\begin{array}{l}
(3,3,3,2,-1,-1,-1), (3,3,3,-1,-1,-2), (3,3,2,-1,-1,-3),\\
(3,3,-1,-2,-3), (3,2,-1,-3,-3),\\
(5,3,3,-1,-1), (5,3,-1,-3), (3,3,-1,-5).
\end{array}
$$

Pretzel diagrams of types $3.2$ and $3.3$ have span equal to 2 and 1, respectively, so they do not belong to $\mathcal L_{10}$.

\item If $P$ is of type 3.4 then $(r,\lambda,s)=(2,-1,0)$, $a_2=2$, $a_1>3$ and $S=a_1-1$. The only diagram is
$$
(11,2,-1).
$$

\item If $P$ is of type 3.5 then $(r,\lambda,s)=(2,-1,0)$, $a_2>2$ and $S=\Sigma-2=a_1+a_2-2$. Decompositions of $\Sigma=12$ with two numbers are $3+9$ and $5+7$. We obtain
$$
(9,3,-1), (7,5,-1).
$$

Pretzel diagrams of type $4$ are either trivial or torus links with 2 strands, which have representative pretzel diagram of type $1$.  
\end{itemize}

In all diagrams of type $5$, $r=1$, $\lambda =0$ and $s>1$, so they have at least three entries, only the first one positive, none of them $\pm1$.

\begin{itemize}
\item If $P$ is of type 5.1 then $a_1\not= -a_2-1$ and $S=\Sigma- \min \{a_1, -a_2-1 \}$. We distinguish two possibilities:

If $a_1>-a_2-1$ then $S=\Sigma- (-a_2-1)$ hence $S-1=a_1-a_3-\dots$ The decompositions of $S-1=9$ with at least two numbers are $2+7$, $3+6$, $3+3+3$ and $4+5$. Assigning signs (recall that $r=1$) and taking into account the restrictions on $a_2$ we obtain 
$$
\begin{array}{l}
(3, -3 ,-6), (6, -3,-3), \\
(3, -2,-3,-3), \\
(3,-3,-3,-3), \\
(4, -3,-5), (5, -3 ,-4).
\end{array}
$$

If $a_1<-a_2-1$ then $S=\Sigma- a_1$, i.e., $S =-a_2-a_3-\dots$ The only decomposition of $S=10$ with at least two numbers and all of them greater than 3 (since $-a_2>a_1+1\geq 3$) is $5+5$. Assigning signs ($r=1$) and taking into account the restriction $2\leq a_1<-a_2-1$, we obtain
$$
(2 ,-5,-5), \quad ( 3, -5,-5).
$$

\item If $P$ is of type 5.2 then $a_2=-a_1-1$, $a_3\not= a_2-1 (=-a_1-2)$ and $S= \Sigma- \min \{a_1+1, -a_3-1\}$. Again we distinguish two possibilities:

If $a_3>-a_1-2$ then $S= \Sigma-(-a_3-1)$ hence $S-1=a_1-a_2-a_4-\dots$. 
The only decomposition of $S-1=9$ with at least two numbers, and two of them consecutive (since $-a_2=a_1+1$), is $4+5$. Assigning signs (recall that $r=1$) and having into account the restriction $a_2\geq a_3>a_2-1$, i.e., $a_3=a_2$, we obtain the only possibility
$$
(4,-5,-5).
$$

If $a_3<-a_1-2$ then $S= \Sigma-(a_1+1)=\Sigma+a_2$ hence $S =a_1-a_3-\dots$. 
The only decomposition of $S=10$ with at least two numbers, the difference between the two smallest ones greater than 2 (since $-a_3>a_1+2$), is $3+7$. Assigning signs and adding $a_2=-a_1-1$, we obtain 
$$
(3, -4, -7).
$$
 
\item If $P$ is of type 5.3 then $P=P(a_1, -a_1-1, -a_1-2)$. Exceptional cases in this item have two even entries, so they do not define knots. Since $10=S=2a_1$, the only possibility is
$$
(5, -6, -7).
$$

\item If $P$ is of type 5.4, $P=P(a_1, -a_1-1, -a_1-2, a_4, \dots)$ with $a_4<-a_1-3$. As before, exceptional cases do not define knots. Notice that $a_1$ must be odd, otherwise there would be two even entries. So $a_2$ is even. In this case, the theorem gives $S=\Sigma-a_1-3$, so $S+3=-a_2-a_3-a_4-\dots$. Thus we have to decompose $S+3=13$ as sum of at least 3 positive numbers, $a\leq b\leq c$ with $a>3$, $b=a+1$ and $c\geq a+2$. But this is not possible.
 
All diagrams of type $5.5$ have two even entries, so they are not in ${\cal L}_{10}$.

\item If $P$ is of type 5.6 then $P=P(a_1, -a_1-1, -a_1-2, -a_1-2, \dots )$ with span $S=\Sigma-a_1-1$. Thus $S+1=-a_2-a_3-a_4-\dots$ and we have to decompose $S+1 = 11$ as sum of at least 3 positive numbers $a, b, c$ with $a+1=b=c$, $a\geq 3$ and $b$ odd. But this is not possible.

Pretzel diagrams of types 6.1 and 6.2 are either trivial or torus links with two strands, which have representative pretzel diagrams of type 1.

\item If $P$ is of type 6.3 then $P=P(1, a_2, a_3)$ with $a_2, a_3$ negative, $a_2\not =-2$, and $S=\Sigma-2$. The decompositions of $S+2=12$ with two numbers greater than 2 are $3+9$ and $5+7$. We have diagrams 
$$
(1,-3,-9), \quad (1,-5,-7).
$$

\item If $P$ is of type  6.4 then $P=P(1, a_2, a_3, a_4, \dots)$ with $a_2<-2$ and $S=\Sigma-1$.  The only decomposition of $S+1=11$ with at least three numbers greater than 2 is $3+3+5$, which gives the diagram
$$
(1,-3,-3,-5). 
$$

Finally, since $P(1,-2,a_3, \dots, a_n)$ and $P(2,a_2, \dots, a_n)$ are isotopic by Remark 7.5, and $r=1$ and $\lambda =0$ for $P(2,a_3, \dots, a_n)$, any pretzel link represented by a diagram of type 7 has also a diagram of type 4 or type 5. It follows that type 7 diagrams are not necessary in order to fill a complete set ${\cal L}_{10}$. This completes the proof.
\end{itemize}
\end{proof}

\section{Appendix} \label{SectionAppendix}
Here we detail the proof of Theorem \ref{TheoremSpan}.

\begin{proof}
Let $P=P(a_1, \ldots ,a_n)$ and $\epsilon _i=(-1)^{a_i}$ for every $i\in \{ 1,\ldots , n\}$. 
Indeed, 
\begin{eqnarray}\label{relationJK}
\tn{span}_tV(L) = \frac{1}{4} \tn{span}_A \langle P\rangle.
\end{eqnarray}
 By Theorem~\ref{TheoremMain} and Lemma~\ref{lemma 0}
\begin{eqnarray*}
\delta ^n \langle P\rangle 
= \prod_{i=1}^{n} A^{-3a_i} \left(\prod_{i=1}^{n}(\epsilon _i + (\delta^2 - 1)A^{4a_{i}}) \right. \left.
+(\delta^2 - 1)\prod_{i=1}^{n}(\epsilon_i - A^{4a_{i}}) \right).
\end{eqnarray*}

Let $B=A^4$. Then $\delta ^2 - 1 = B^{-1}+1+B$ and
we have that
\begin{eqnarray}\label{relationB}
4n + \tn{span}_A \langle P\rangle
= \tn{span}_A (\delta ^n \langle P\rangle) 
= 4\,\tn{span}_B(p(B))
\end{eqnarray}
where
\begin{eqnarray}
p(B) = \prod_{i=1}^{n}
\left( \epsilon_i + (B^{-1} + 1 + B) B^{a_{i}} \right) + (B^{-1}+1+B) \prod_{i=1}^{n}(\epsilon_i - B^{a_{i}}).
\end{eqnarray}

Putting together \ref{relationJK} and \ref{relationB}, we deduce that
\begin{eqnarray}\label{relationVB}
\tn{span}_t V(L) = \tn{span}_B(p(B)) - n.
\end{eqnarray}
In the rest of the proof we fix our attention in the polynomial $p(B)$ and calculate its span in the variable $B$. Let $F(B) = \prod_{i=1}^{n}(\epsilon _i + (B^{-1} + 1 + B) B^{a_{i}})$ and $S(B) = (B^{-1}+1+B) \prod_{i=1}^{n}(\epsilon _i - B^{a_{i}})$ be the first and second summands of $p(B)$. Let $h_F$ and $l_F$ be respectively the highest and lowest degree of $F(B)$, and let
$h_S$ and $l_S$ be respectively the highest and lowest degree of $S(B)$. Let $h$ and $l$ be respectively the highest and lowest degree of $p(B)$. By definition $\tn{span}_B(p(B)) = h - l$, and clearly $h = \max \{h_F, h_S\}$ if $h_F \neq h_S$ and $l = \min\{l_F,l_S\}$ if $l_F \neq l_S$. The strategy will be then to calculate $h_F$, $l_F$, $h_S$ and $l_S$, and whenever $h_F = h_S$ look carefully at the possible cancellations of the highest degree terms of the
summands $F(B)$ and $S(B)$. Under the hypothesis of the theorem $l_F$ and $l_S$ will be found to be different.
Suppose that $z>0$. Then 
$p(B) = F(B)$ and 
\begin{eqnarray*}
\tn{span}_B(p(B))& = &\sum_{i=1}^{n}\tn{span}_B(\epsilon _i + (B^{-1}+1+B)B^{a_{i}}) \\
            &=& \sum_{|a_{i}|>1}(|a_{i}|+1) + \sum_{|a_{i}|=1}1
+\sum_{|a_{i}|=0}2 \\
            &=&\sum_{|a_{i}|>1}|a_{i}| + r + s + \alpha + \beta + 2z \\
            &=& \sum_{|a_{i}|>1}|a_{i}| + n + z
\end{eqnarray*}
and item {\it 1} follows.

Assume now that $z=0$. It is easy to see that
$$
\begin{array}{ll}
h_{F} = r+ \sum_{a_{l}>1}a_{l} +2\alpha -\beta, &
l_{F} = -s+\sum_{a_{j}<-1}a_{j}+\alpha -2\beta,\\ &\\
h_{S} = 1+\sum_{a_{l}>1}a_{l} +\alpha, &
l_{S} = -1+\sum_{a_{j}<-1}a_{j} -\beta,
\end{array}
$$
hence  
$$
h = \left\{ \begin{array}{lcc}
  h_F & \tn{if} & r+\lambda >1 \\ &&\\
  h_S & \tn{if} & r+\lambda <1 
\end{array}\right.,
\qquad 
l = \left\{ \begin{array}{lcc}
  l_F & \tn{if} & s-\lambda >1 \\ &&\\
  l_S & \tn{if} & s-\lambda <1
\end{array}\right.
$$
and item {\it 2} follows (if $r + \lambda < 1$ and $s - \lambda <1$ then $r=s=\lambda = 0$ and, since $P$ is reduced, $P=(1, -1)$ which represents the split union of two trivial knots whose Jones polynomial has span two).

We now prove item {\it 3.1}. Assume $r + \lambda =1$, $r>1$ and $(r, \lambda, s) \not= (2,-1,0)$. First note that $s - \lambda > 1$ hence $l = l_F$. By Remark \ref{obs}.\ref{obs1} we may assume that $a_l > 1$, $l =1,\ldots, r$ and $a_j<-1$, $j = r + 1, \ldots, r+s$. We have that $h_F = h_S$, 
$$
\begin{array}{rcl}
F(B) & = & \epsilon_{r+1}\ldots \epsilon_{r+s}B^{h_F} +(r+\alpha + \beta + d) \epsilon_{r+1}\ldots \epsilon_{r+s} B^{h_F-1} \\
&& + \tn{ monomials of degree $< h_F-1$}
\end{array}
$$
and
$$
\begin{array}{rcl}
S(B) & = & (-1)^r \epsilon_{r+1}\ldots \epsilon_{r+s} (-1)^{\alpha} (-1)^{\beta} B^{h_S} - (\alpha + \beta + 1) \epsilon_{r+1}\ldots \epsilon_{r+s} B^{h_S-1} \\
&& + \tn{ monomials of degree $< h_S-1$}
\end{array}
$$
where $d$ is the number of $a_k=-2$ (note that
$\epsilon_{k}=+1$ if $a_k=-2$). Since $r + \lambda~=~1$ the first summands cancel in $p(B)$. Since $r>1$, $(r-1+d) \epsilon_{r+1} \cdots \epsilon_{r+s} B^{h_F-1}$ is the highest degree term of $p(B)$. Hence $h =h_F - 1$ and {\it 3.1} follows.

The exceptional cases {\it 3.2} to {\it 3.5}, when $(r,\lambda, s) = (2,-1,0)$, will be proved later by using items {\it 6} and {\it 7}.

In the rest of the proof we will write $P\sim P'$ if both diagrams $P$ and $P'$ define the same link. 

We now prove item {\it 4.} Item {\it 4.1} corresponds to the trivial knot $P(a_1)$, with span~$1$. Since $P(a,b)\sim P(0, a+b)$, in the item {\it 4.3} we have also the trivial knot with span one, and for items {\it 4.2} and {\it 4.4} we may apply item {\it 1}.

Item {\it 5} is the difficult one, and we leave it for last. We now prove item {\it 6}:
\begin{itemize}
\item[{\it 6.1.}] It is $P(1)$, the trivial knot, with $S=1$. 
\item[{\it 6.2.}] Since $P(1,a_2) \sim P(0,1+a_2)$, then $S=1+|1+a_2| = 1-(1+a_2)=-a_2$ by item {\it 1} (note that $|1+a_2|>1$ since $a_2<-2$).
\item[{\it 6.3.}] The diagram is then $P(1,a_2,a_3)$. By expanding $p(B)$ we find that $\tn{span}_B(p(B))= h - l = 1- (a_2+a_3)$.
\item[{\it 6.4.}] It is the case $P(1, b_{1},\dots ,b_{s})$ with $s>2$ and $b_j<-2$, $j=1, \ldots ,s$. Then 
$$
\begin{array}{c}
p(B) = (B+B^2)\prod_{j=1}^{s} (B^{b_j-1} + B^{b_j}
+ B^{b_{j}+1} + \epsilon_j) \\
\qquad \qquad + (B^{-1}+1+B)(-1-B)\prod_{j=1}^{s}(-B^{b_{j}}
+ \epsilon_j)
\end{array}
$$
has span $s - b_1- \ldots -b_s$. 
\end{itemize}

Remark \ref{obs}.\ref{obs5} is used constantly along the proof of item {\it 7}. Items {\it 7.1} to {\it 7.4} are derived from items {\it 4.1} to {\it 4.4} respectively. Items {\it 7.5} to {\it 7.15} are derived from item {\it 5}.
\begin{itemize}
\item[{\it 7.1.}] $P(1,-2)\sim P(2)$, the trivial knot, $S=1$.
\item[{\it 7.2.}] $P(1,-2,-2)\sim P(2,-2)$. Then $S=2$ by item {\it 4.2}.
\item[{\it 7.3.}] $P(1,-2,-3)\sim P(2,-3)$. Then $S=1$ by item {\it 4.3}.
\item[{\it 7.4.}] $P(1,-2,a_3)\sim P(2,a_3)$ with $a_3<-3$. Then $S=1+|2+a_3| = 1-(2+a_3)=-1-a_3$ by item {\it 4.4}.
\item[{\it 7.5.}] $P(1,-2,-2,a_4)\sim P(2,-2,a_4)$ and we use item {\it 5.1}.
\item[{\it 7.6.}] $P(1,-2,-3,-3) \sim P(2,-3,-3)$ and we use item {\it 5.2}.
\item[{\it 7.7.}] $P(1,-2,-3,-4) \sim P(2,-3,-4)$ and we use the exceptional case of item {\it 5.3}.
\item[{\it 7.8.}] By item {\it 5.2.}
\item[{\it 7.9.}] By item {\it 5.1}.
\item[{\it 7.10.}] By item {\it 5.1.}
\item[{\it 7.11.}] By item {\it 5.2.}
\item[{\it 7.12.}] By item {\it 5.6.}
\item[{\it 7.13.}] By item {\it 5.5.}
\item[{\it 7.14.}] By item {\it 5.4.}
\item[{\it 7.15.}] By item {\it 5.2.}
\item[{\it 7.16.}] By item {\it 5.1.}
\end{itemize}

Exceptional cases of item {\it 3} correspond to pretzel diagrams $P(a_1, a_2, -1)$ with $a_1 \geq a_2 \geq 2$. After taking the mirror image and reordering the entries we obtain $P(1, b_2, b_3) = P(1,-a_2,-a_1)$ with $-2\geq -a_2 \geq -a_1$, hence $(r', \lambda', s') = (0, 1, 2)$.
\begin{itemize}
\item[{\it 3.2}] By item {\it 7.2} since if $a_2=2$ and $a_1=2$ then $b_2 = b_3 = -2$.
\item[{\it 3.3}] By item {\it 7.3} since if $a_2=2$ and $a_1=3$ then $b_2= -2$ and $b_3 = -3$.
\item[{\it 3.4}] By item {\it 7.4} since if $a_2=2$ and $a_1 >3$ then $b_2 = -2$ and $b_3 < -3$.
\item[{\it 3.5}] By item {\it 6.3} since if $a_2>2$ then $b_2\not= -2$.
\end{itemize}

Finally we concentrate in the difficult cases, item {\it 5.} Since $r + \lambda = 1$, $r = 1$ and $s > 1$ it follows that $s - \lambda > 1$ and
$l = l_F$. But as in item {\it 3.1}, $h_F = h_S$ and the terms with this degree cancel. Moreover, other previous terms cancel too. The point is how many steps we have to go down in order to find the first no cancellation. This require some laborious calculations, that are shown carefully afterwards (for the item {\it 5.4} the case $a_1=2$ must be considered separately). In this process more and more inner coefficients have to be considered, having the impression that the process has not an end. But it has!

Recall that $\epsilon_i = (-1)^{a_i}$ for every $i \in \{1, \dots, n\}$, $a_1 > 1$, $a_j < - 1$ if $j \in \{2, \dots, n\}$ and $|a_2| \leq \ldots \leq |a_n|$. Also $\tn{span}_tV(L)=\tn{span}_B(p(B)) - n$ where $p(B)=F(B)+S(B)$. It will be convenient to write both polynomials $F(B)$ and $S(B)$ in such a way that the degree of the monomials in each factor increases when going right:
$$
F(B) 
= (\epsilon_1 + B^{a_{1}-1} + B^{a_1} + B^{a_1+1}) \prod_{j=2}^{n}(B^{a_{j} - 1} + B^{a_j} + B^{a_j+1} + \epsilon _j)
$$
and
$$
\begin{array}{rcl}
S(B) & = & (B^{-1}+1+B)(\epsilon _1-B^{a_{1}})
\prod_{j=2}^{n}(-B^{a_{j}}+\epsilon _j) \\ &&\\
& = & (\epsilon _1 B^{-1} + \epsilon _1 + \epsilon_1 B - B^{a_1 - 1} - B^{a_1} - B^{a_1+1}) 
\prod_{j=2}^{n}(-B^{a_{j}}+\epsilon _j).
\end{array}
$$

The proof of item {\it 5.3} is direct: it is enough to expand $P(a_1, -a_1-1, -a_1-2)$ to find that $\tn{span}_B(p(B))=2a_1+3$ if $a_1\not=2$, and it is $6$ if $a_1=2$.

Under the hypothesis of item {\it 5} we have that $h_F- l_F = \sum_{|a_i|>1} |a_i| + n$, so it is enough to show that the value of $h$ is the following for each case of item {\it 5}:
\begin{itemize}
\item[{\it 5.1.}] $h = 
\left\{ \begin{array}{l}
h_F - a_1 = 1 \tn{ if $a_1 < |a_2| - 1$},\\
h_F - (|a_2|-1) = 2 + a_1 + a_2 \tn{ if $a_{1}>|a_{2}|-1$}.
\end{array} \right.$
\end{itemize}

$\bullet$ Suppose first that $a_1 < |a_2|-1$. If in $F$ we take a $B^{a_j+1}$ of one of the last $n-1$ factors, we obtain degree $<1$ since $(a_1+1)+(a_2+1) <1$. If in $S$ we take a $-B^{a_j}$ of one of the last $n-1$ factors, we obtain degree $<0$ since $(a_1+1) + a_2 <0$. Then 
$$
F = \epsilon_2\cdots \epsilon_n B^{a_1+1} + \epsilon_2\cdots \epsilon_nB^{a_1}
+ \epsilon_2\cdots \epsilon_nB^{a_1-1}
+ \tn{ monomials of degree $<1$}
$$
and
$$
\begin{array}{l}
S = -\epsilon_2\cdots \epsilon_n B^{a_1+1} - \epsilon_2\cdots \epsilon_nB^{a_1}
- \epsilon_2\cdots \epsilon_nB^{a_1-1} + \epsilon_1\cdots \epsilon_nB \\
\qquad + \tn{ monomials of degree $<1$} 
\end{array}
$$
and $F+S = \epsilon_1\cdots \epsilon_nB + \tn{ monomials of degree $<1$}$, hence $h=1=h_F-a_1$.

$\bullet$ Suppose that $a_1 > |a_2|-1$ and let $h_0 = 2 + a_1 + a_2 = h_F - (|a_2|-1)$. If in $S$ we take a $-B^{a_j}$ of one of the last  $n-1$ factors, we obtain degree $<h_0$ since $(a_1+1) + a_2 < h$. Then, noting that $h_0>1$, we have that
$$
\begin{array}{l}
F = \epsilon_2\cdots \epsilon_n B^{a_1+1} + \epsilon_2\cdots \epsilon_nB^{a_1} + \epsilon_2\cdots \epsilon_nB^{a_1-1} + (1+k)\epsilon_3\cdots \epsilon_nB^{a_1+1+a_2+1} \\ 
\qquad + \tn{ monomials of degree $< h_0$} 
\end{array}
$$
if $a_2 = a_3 = \ldots  = a_{k+2} > a_{k+3}$ (note that if $a_2=a_3$, $\epsilon_2=\epsilon_3$ hence $\epsilon_3\cdots \epsilon_n = \epsilon_2 \epsilon_4\cdots \epsilon_n$,  etc.) and  
$$
\begin{array}{l}
S = -\epsilon_2\cdots \epsilon_n B^{a_1+1} - \epsilon_2\cdots \epsilon_nB^{a_1}
- \epsilon_2\cdots \epsilon_nB^{a_1-1} \\
\qquad + \tn{ monomials of degree $< h_0$}. 
\end{array}
$$
It follows that $F+S = (1+k) \epsilon_3\cdots \epsilon_n B^{h_0}
+ \tn{ monomials of degree } < h_0$, hence $h = h_0 = 2+a_1+a_2$ and item {\it 5.1} is proved. 

\begin{itemize}
\item[{\it 5.2.}] $h = 
\left\{ \begin{array}{l}
h_F - |a_2| = 0  \tn{ if $a_1 = |a_2|-1$ and $|a_2|<|a_3|-1$}, \\
h_F - a_1 = 1 \tn{ if $a_1=|a_2|-1$ and $|a_2|=|a_3|$}.
\end{array} \right.$
\end{itemize}

$\bullet$ Suppose that $a_1 = |a_2|-1$ and $|a_2|<|a_3|-1$, hence $a_1+a_2 +1 = 0$ and $a_1+a_3+2<0$. Then $h = h_F - |a_2| = 1 +a_1 + a_2 = 0$. Then
$$
\begin{array}{l}
F = \epsilon_2\cdots \epsilon_n B^{a_1+1} + \epsilon_2\cdots \epsilon_nB^{a_1}
+ \epsilon_2\cdots \epsilon_nB^{a_1-1} + \color{red}{\epsilon_1\cdots \epsilon_n}\color{black}  \\
\qquad + \epsilon_3\cdots \epsilon_n B^{{a_1+1+a_2+1}\color{blue}{=1}\color{black}} + \epsilon_3\cdots \epsilon_n B^{a_1+a_2+1 \color{red}{=0}\color{black}} \\ 
\qquad + \epsilon_3\cdots \epsilon_n B^{a_1+1+a_2 \color{red}{=0}\color{black}} + \tn{ monomials of degree $< 0$} 
\end{array}
$$
and
$$
\begin{array}{l}
S = -\epsilon_2\cdots \epsilon_n B^{a_1+1} - \epsilon_2\cdots \epsilon_nB^{a_1}
- \epsilon_2\cdots \epsilon_nB^{a_1-1} + \color{blue}{\epsilon_1\cdots \epsilon_nB }\color{black} 
+ \color{red}{\epsilon_1\cdots \epsilon_n}\color{black} \\ 
\qquad + \epsilon_3\cdots \epsilon_nB^{a_1+1+a_2\color{red}{=0}\color{black}} + \tn{ monomials of degree $< 0$}. 
\end{array}
$$
Since $\epsilon_1\epsilon_2 = -1$ we have that $F+S = \epsilon_3\cdots \epsilon_n  + \tn{ monomials of degree $< 0$}$ and $h = 0 = 1 +a_1 + a_2 = h_F - |a_2|$.

$\bullet$ Suppose that $a_1 = |a_2|-1$ and $|a_2|=|a_3|$ hence $a_1+a_2 +1 = 0$. If in $F$ we take two $B^{a_j+1}$ of the last $n-1$ factors, we obtain degree $<1$ since $(a_1+1)+(a_2+1)+(a_3+1) = a_3 + 2 < 1$. If in $S$ we take a $-B^{a_j}$ of one of the last $n-1$ factors, we obtain degree $<1$ since $(a_1+1) + a_2 = 0$. Then
$$
\begin{array}{l}
F = \epsilon_2\cdots \epsilon_n B^{a_1+1} + \epsilon_2\cdots \epsilon_nB^{a_1}
+ \epsilon_2\cdots \epsilon_nB^{a_1-1}\\
\qquad + (2+k) \epsilon_3\cdots \epsilon_n B^{{a_1+1+a_2+1}\color{red}{=1}\color{black}} + \tn{ monomials of degree $<1$} 
\end{array}
$$
if $a_2 = a_3 = \ldots = a_{k+3} > a_{k+4}$ and
$$
\begin{array}{l}
S = -\epsilon_2\cdots \epsilon_n B^{a_1+1} - \epsilon_2\cdots \epsilon_nB^{a_1}
- \epsilon_2\cdots \epsilon_nB^{a_1-1} + \color{red}{\epsilon_1\cdots \epsilon_nB}\color{black} \\ 
\qquad + \tn{ monomials of degree $<1$}. 
\end{array}
$$
Since $\epsilon_1\epsilon_2=-1$ we have that $F+S = (1+k) \epsilon_3\cdots \epsilon_nB + \tn{ monomials of degree $< 1$}$ and $h = 1 = h_F - a_1$.


\begin{itemize}
\item[{\it 5.4.}] $h = 
\left\{ \begin{array}{l} 
h_F-(|a_3|+1) = - 2 \\
\qquad \tn{ if $a_1=|a_2|-1$, $|a_2|=|a_3|-1$, $|a_3|<|a_4|-1$ and $a_1 >2$}\\ 
\qquad \tn{ or $(a_1, a_2, a_3, a_4) = (2, -3, -4, -6)$},\\ 
-3 \tn{ if $(a_1, a_2, a_3, a_4) = (2,-3,-4,a_4)$ with $a_4<-6$.}
\end{array}\right.$
\end{itemize}

$\bullet$ Suppose that $a_1 = |a_2|-1$, $|a_2|=|a_3|-1$, $|a_3|<|a_4|-1$ and $a_1>2$, hence $a_1+ a_2 + 1 = 0$, $a_1+a_3+2=0$ and $a_1+a_4+3<0$. If in $F$ we take two $B^{a_j+1}$ of the last  $n-1$ factors we obtain degree $< -2$ since $(a_1+1)+(a_2+1)+(a_3+1) = -a_1 < -2$. If in $S$ we take two $-B^{a_j}$ of the last $n-1$ factors we obtain degree $<-2$ since $(a_1+1) + a_2 + a_3 = a_3 < -2$. Then
$$
\begin{array}{l}
F = \epsilon_2\cdots \epsilon_n B^{a_1+1} + \epsilon_2\cdots \epsilon_nB^{a_1}
+ \epsilon_2\cdots \epsilon_nB^{a_1-1} + \color{brown}{\epsilon_1\cdots \epsilon_n}\color{black}  \\
\qquad + \epsilon_3\cdots \epsilon_n B^{{a_1+1+a_2+1}\color{cyan}{=1}\color{black}} 
+ \epsilon_2\epsilon_4\cdots \epsilon_n B^{{a_1+1+a_3+1}\color{brown}{=0}\color{black}} \\
\qquad \qquad + (1+k) \epsilon_2\epsilon_3\epsilon_5\cdots \epsilon_n B^{{a_1+1+a_4+1}\color{red}{\leq -2}\color{black}} \\
\qquad + \epsilon_3\cdots \epsilon_n B^{a_1+a_2+1 \color{brown}{=0}\color{black}} 
+ \epsilon_2\epsilon_4\cdots \epsilon_n B^{a_1 + a_3 + 1 \color{blue}{=-1}\color{black}} \\ 
\qquad + \epsilon_3\cdots \epsilon_n B^{a_1 - 1 + a_2 + 1 \color{blue}{=-1}\color{black}} 
+ \epsilon_2\epsilon_4\cdots \epsilon_n B^{a_1 - 1 + a_3 +1\color{red}{=-2}\color{black}} \\
\qquad  + \epsilon_3\cdots \epsilon_n B^{a_1 + 1 + a_2 \color{brown}{=0}\color{black}} 
+ \epsilon_2\epsilon_4\cdots \epsilon_n B^{a_1 + 1 + a_3 \color{blue}{=-1}\color{black}} \\
\qquad + \epsilon_3\cdots \epsilon_n B^{a_1 + a_2 \color{blue}{=-1}\color{black}} 
+ \epsilon_2\epsilon_4\cdots \epsilon_n B^{a_1+a_3 \color{red}{=-2}\color{black}} \\
\qquad + \epsilon_3\cdots \epsilon_n B^{a_1-1+a_2 \color{red}{=-2}\color{black}} \\
\qquad + \epsilon_3\cdots \epsilon_n B^{a_1 + 1 + a_2 -1 \color{blue}{=-1}\color{black}} 
+ \epsilon_2\epsilon_4\cdots \epsilon_n B^{a_1 + 1 + a_3 -1 \color{red}{=-2}\color{black}} \\
\qquad + \epsilon_3\cdots \epsilon_n B^{a_1+a_2-1 \color{red}{=-2}\color{black}} \\
\qquad + \tn{ monomials of degree $<-2$}
\end{array}
$$
if $a_4 = a_5 = \ldots = a_{k+4} > a_{k+5}$ and
$$
\begin{array}{l}
S = -\epsilon_2\cdots \epsilon_n B^{a_1+1} - \epsilon_2\cdots \epsilon_nB^{a_1}
- \epsilon_2\cdots \epsilon_nB^{a_1-1} + \color{cyan}{\epsilon_1\cdots \epsilon_n}B  \\
\qquad + \color{brown}{\epsilon_1\cdots \epsilon_n}\color{black}
+ \color{blue}{\epsilon_1\cdots \epsilon_nB^{-1}}\color{black} \\ 
\qquad + \epsilon_3\cdots \epsilon_nB^{a_1+1+a_2\color{brown}{=0}\color{black}} 
+ \epsilon_2\epsilon_4\cdots \epsilon_nB^{a_1+1+a_3\color{blue}{=-1}\color{black}} \\
\qquad + \epsilon_3\cdots \epsilon_nB^{a_1+a_2\color{blue}{=-1}\color{black}} 
+ \epsilon_2\epsilon_4\cdots \epsilon_nB^{a_1+a_3\color{red}{=-2}\color{black}} \\
\qquad + \epsilon_3\cdots \epsilon_nB^{a_1-1+a_2\color{red}{=-2}\color{black}} \\
\qquad + \tn{ monomials of degree $<-2$}. 
\end{array}
$$
Since $\epsilon_1 = -\epsilon_2 = \epsilon_3$ all terms of degree greater than $-2$ cancel. For degree $-2$, the two addends provided by $S$ cancel, and the last four addends provided by $F$ cancel. Then if $a_1+1+a_4+1 < -2$ we have $F+S=\epsilon_2\epsilon_4\cdots \epsilon_n B^{-2}$; otherwise $a_1+1+a_4+1 = -2$ hence $\epsilon_3 = \epsilon_1=\epsilon_4$ and it follows that $F+S = (2+k)\epsilon_2\epsilon_3\epsilon_5\cdots \epsilon_n B^{-2}$. In both cases $h=-2=h_F-(|a_3|+1)$.

$\bullet$ Suppose that $(a_1, a_2, a_3, a_4) = (2, -3, -4, -6)$ hence $a_1+ a_2 + 1 = 0$, $a_1+a_3+2=0$, $a_1+a_4+4 = 0$, $\epsilon_1 = \epsilon_3 = 1$ and $\epsilon_2 = -1$. If in $F$ we take two $B^{a_j+1}$ other than $B^{a_2+1}$ and $B^{a_3+1}$ among the last $n-1$ factors, we obtain degree $<-2$ since $(a_1+1)+(a_2+1)+(a_4+1)=-4$ while $(a_1+1)+(a_2+1)+(a_3+1)=-2$. 
Instead, if in $S$ we take two $-B^{a_j}$ of the last $n-1$ factors we obtain degree $<-2$ since $(a_1+1) + a_2 + a_3 = -4 < -2$. Then
$$
\begin{array}{l}
F = \epsilon_2\cdots \epsilon_n B^{a_1+1} + \epsilon_2\cdots \epsilon_nB^{a_1}
+ \epsilon_2\cdots \epsilon_nB^{a_1-1} + \color{brown}
{\epsilon_1\cdots \epsilon_n}\color{black}  \\
\qquad + \epsilon_3\cdots \epsilon_n B^{{a_1+1+a_2+1}\color{cyan}{=1}\color{black}} 
+ \epsilon_2\epsilon_4\cdots \epsilon_n B^{{a_1+1+a_3+1}\color{brown}{=0}\color{black}} \\
\qquad \qquad  
+ (1+k) \epsilon_2\epsilon_3\epsilon_5\cdots \epsilon_n B^{{a_1+1+a_4+1}\color{red}{= -2}\color{black}} \\
\qquad  + \epsilon_3\cdots \epsilon_n B^{a_1+a_2+1 \color{brown}{=0}\color{black}} 
+ \epsilon_2\epsilon_4\cdots \epsilon_n B^{a_1 + a_3 + 1 \color{blue}{=-1}\color{black}} \\ 
\qquad  + \epsilon_3\cdots \epsilon_n B^{a_1 - 1 + a_2 + 1 \color{blue}{=-1}\color{black}} 
+ \epsilon_2\epsilon_4\cdots \epsilon_n B^{a_1 - 1 + a_3 +1\color{red}{=-2}\color{black}} \\
\qquad  + \epsilon_1 \epsilon_3\cdots \epsilon_n B^{0 + a_2 + 1 \color{red}{=-2}\color{black}} \\
\qquad  + \epsilon_3\cdots \epsilon_n B^{a_1 + 1 + a_2 \color{brown}{=0}\color{black}} 
+ \epsilon_2\epsilon_4\cdots \epsilon_n B^{a_1 + 1 + a_3 \color{blue}{=-1}\color{black}} \\
\qquad  + \epsilon_3\cdots \epsilon_n B^{a_1 + a_2 \color{blue}{=-1}\color{black}} 
+ \epsilon_2\epsilon_4\cdots \epsilon_n B^{a_1+a_3 \color{red}{=-2}\color{black}} \\
\qquad  + \epsilon_3\cdots \epsilon_n B^{a_1-1+a_2 \color{red}{=-2}\color{black}} \\ 
\qquad  + \epsilon_3\cdots \epsilon_n B^{a_1 + 1 + a_2 -1 \color{blue}{=-1}\color{black}} 
+ \epsilon_2\epsilon_4\cdots \epsilon_n B^{a_1 + 1 + a_3 -1 \color{red}{=-2}\color{black}} \\
\qquad  + \epsilon_3\cdots \epsilon_n B^{a_1+a_2-1 \color{red}{=-2}\color{black}} \\
\qquad  + \epsilon_4\cdots \epsilon_n B^{a_1+1+a_2+1 +a_3+1\color{red}{=-2}\color{black}} \\
\qquad  + \tn{ monomials of degree $<-2$}
\end{array}
$$
if $a_4 = a_5 = \ldots = a_{k+4} > a_{k+5}$, and 
$$
\begin{array}{l}
S = -\epsilon_2\cdots \epsilon_n B^{a_1+1} - \epsilon_2\cdots \epsilon_nB^{a_1}
- \epsilon_2\cdots \epsilon_nB^{a_1-1} 
+ \color{cyan}{\epsilon_1\cdots \epsilon_n}B  \\
\qquad \qquad + \color{brown}{\epsilon_1\cdots \epsilon_n}\color{black}
+ \color{blue}{\epsilon_1\cdots \epsilon_nB^{-1}}\color{black}\\ 
\qquad + \epsilon_3\cdots \epsilon_n B^{a_1+1+a_2\color{brown}{=0}\color{black}} 
+ \epsilon_2\epsilon_4\cdots \epsilon_n B^{a_1+1+a_3 \color{blue}{=-1}\color{black}} \\
\qquad + \epsilon_3\cdots \epsilon_nB^{a_1+a_2\color{blue}{=-1} \color{black}} 
+ \epsilon_2\epsilon_4\cdots \epsilon_nB^{a_1+a_3\color{red}{=-2}\color{black}} \\
\qquad + \epsilon_3\cdots \epsilon_nB^{a_1-1+a_2\color{red}{=-2}\color{black}} \\
\qquad - \epsilon_1\epsilon_3\cdots \epsilon_nB^{1+a_2\color{red}{=-2}\color{black}} \\
\qquad + \tn{ monomials of degree $<-2$}. 
\end{array}
$$
Monomials of degree greater that $-2$ cancel, and it turns out that   $F+S = (1+k) \epsilon_2\epsilon_3\epsilon_5\cdots \epsilon_n B^{-2}$ hence 
$h = -2 =h_F-(|a_3|+1)$.
 
$\bullet$ If $(a_1, a_2, a_3, a_4)=(2,-3,-4,a_4)$ with $a_4 < -6$ then  $a_1+a_2+1=0$, $a_1+a_3+2=0$, $a_1+a_4 < - 4$,  
$\epsilon_1 = \epsilon_3 = 1$ and $\epsilon_2 = -1$. In $F$ there are four addends of degree $\geq -3$ formed with two terms of non-extremal degree among the last $n-1$ factors, since $(a_1+1)+(a_2+1)+(a_3+1) = -2$ and
$$
a_1+(a_2+1)+(a_3+1) = 
(a_1+1)+a_2+(a_3+1) = 
(a_1+1)+(a_2+1)+a_3 = -3
$$
but $(a_1+1) + (a_2+1) + (a_4+1) = a_4 + 2 < -6 + 2 = - 4$.
Instead, if in $S$ we take two $-B^{a_j}$ of the last $n-1$ factors we obtain degree $<-3$ since $(a_1+1) + a_2 + a_3 = -4 < -3$. Then
$$
\begin{array}{l}
F = \epsilon_2\cdots \epsilon_n B^{a_1+1} + \epsilon_2\cdots \epsilon_nB^{a_1}
+ \epsilon_2\cdots \epsilon_nB^{a_1-1} + \color{cyan}{\epsilon_1\cdots \epsilon_n}\color{black}  \\
\qquad + \epsilon_3\cdots \epsilon_n B^{{a_1+1+a_2+1}\color{green}{=1}\color{black}} 
+ \epsilon_2\epsilon_4\cdots \epsilon_n B^{{a_1+1+a_3+1}\color{cyan}{=0}\color{black}} \\
\qquad \qquad + (1+k) \epsilon_2\epsilon_3\epsilon_5\cdots \epsilon_n B^{{a_1+1+a_4+1}\color{red}{\leq -3}\color{black}} \\
\qquad  + \epsilon_3\cdots \epsilon_n B^{a_1+a_2+1 \color{cyan}{=0}\color{black}} 
+ \epsilon_2\epsilon_4\cdots \epsilon_n B^{a_1 + a_3 + 1 \color{brown}{=-1}\color{black}} \\ 
\qquad  + \epsilon_3\cdots \epsilon_n B^{a_1 - 1 + a_2 + 1 \color{brown}{=-1}\color{black}} 
+ \epsilon_2\epsilon_4\cdots \epsilon_n B^{a_1 - 1 + a_3 +1\color{blue}{=-2}\color{black}} \\
\qquad  + \epsilon_1 \epsilon_3\cdots \epsilon_n B^{0 + a_2 + 1 \color{blue}{=-2}\color{black}} 
+ \epsilon_1 \epsilon_2\epsilon_4\cdots \epsilon_n B^{0 + a_3 +1\color{red}{=-3}\color{black}} \\
\qquad  + \epsilon_3\cdots \epsilon_n B^{a_1 + 1 + a_2 \color{cyan}{=0}\color{black}} 
+ \epsilon_2\epsilon_4\cdots \epsilon_n B^{a_1 + 1 + a_3 \color{brown}{=-1}\color{black}} \\
\qquad  + \epsilon_3\cdots \epsilon_n B^{a_1 + a_2 \color{brown}{=-1}\color{black}} 
+ \epsilon_2\epsilon_4\cdots \epsilon_n B^{a_1+a_3 \color{blue}{=-2}\color{black}} \\
\qquad  + \epsilon_3\cdots \epsilon_n B^{a_1-1+a_2 \color{blue}{=-2}\color{black}} 
+ \epsilon_2\epsilon_4\cdots \epsilon_n B^{a_1-1+a_3 \color{red}{=-3}\color{black}}\\
\qquad  + \epsilon_1 \epsilon_3\cdots \epsilon_n B^{0 + a_2 \color{red}{=-3}\color{black}}\\ 
\qquad  + \epsilon_3\cdots \epsilon_n B^{a_1 + 1 + a_2 -1 \color{brown}{=-1}\color{black}} 
+ \epsilon_2\epsilon_4\cdots \epsilon_n B^{a_1 + 1 + a_3 -1 \color{blue}{=-2}\color{black}} \\
\qquad  + \epsilon_3\cdots \epsilon_n B^{a_1+a_2-1 \color{blue}{=-2}\color{black}} 
+ \epsilon_2\epsilon_4\cdots \epsilon_n B^{a_1+a_3-1 \color{red}{=-3}\color{black}}\\
\qquad  + \epsilon_3\cdots \epsilon_n B^{a_1-1+a_2-1 \color{red}{=-3}\color{black}} \\
\qquad  + \epsilon_4\cdots \epsilon_n B^{a_1+1+a_2+1 +a_3+1\color{blue}{=-2}\color{black}} \\
\qquad + \epsilon_4\cdots \epsilon_n B^{a_1+a_2+1 +a_3+1\color{red}{=-3}\color{black}} \\
\qquad + \epsilon_4\cdots \epsilon_n B^{a_1+1+a_2+a_3+1\color{red}{=-3}\color{black}} 
+ \epsilon_4\cdots \epsilon_n B^{a_1+1+a_2+1 +a_3\color{red}{=-3}\color{black}} \\
\qquad + \tn{ monomials of degree $<-3$}
\end{array}
$$
if $a_4 = a_5 = \ldots = a_{k+4} > a_{k+5}$, and
$$
\begin{array}{l}
S = -\epsilon_2\cdots \epsilon_n B^{a_1+1} - \epsilon_2\cdots \epsilon_nB^{a_1}
- \epsilon_2\cdots \epsilon_nB^{a_1-1} 
+ \color{green}{\epsilon_1\cdots \epsilon_n}B \\ 
\qquad + \color{cyan}{\epsilon_1\cdots \epsilon_n}\color{black}
+ \color{brown}{\epsilon_1\cdots \epsilon_nB^{-1}}\color{black}\\ 
\qquad  + \epsilon_3\cdots \epsilon_n B^{a_1+1+a_2\color{cyan}{=0}\color{black}} 
+ \epsilon_2\epsilon_4\cdots \epsilon_n B^{a_1+1+a_3 \color{brown}{=-1}\color{black}} \\
\qquad  + \epsilon_3\cdots \epsilon_nB^{a_1+a_2\color{brown}{=-1} \color{black}} 
+ \epsilon_2\epsilon_4\cdots \epsilon_nB^{a_1+a_3\color{blue}{=-2}\color{black}} \\
\qquad + \epsilon_3\cdots \epsilon_nB^{a_1-1+a_2\color{blue}{=-2}\color{black}} 
+ \epsilon_2 \epsilon_4 \cdots \epsilon_n B^{a_1-1+a_3\color{red}{=-3}\color{black}}\\
\qquad - \epsilon_1\epsilon_3\cdots \epsilon_nB^{1+a_2\color{blue}{=-2}\color{black}} 
- \epsilon_1\epsilon_2 \epsilon_4 \cdots \epsilon_n B^{1+a_3\color{red}{=-3}\color{black}}\\
\qquad - \epsilon_1\epsilon_3\cdots \epsilon_nB^{0+a_2\color{red}{=-3}\color{black}} \\
\qquad + \tn{ monomials of degree $<-3$}. 
\end{array}
$$
Note that the monomials of $S$ in the fourth and fifth rows cancel, what corresponds for $a_1=2$) to the fact that the first factor in $S$ can be simplified, being $S = (\epsilon_1B^{-1} + \epsilon_1 - B^{a_1} - B^{a_1+1}) (-B^{a_2} + \epsilon_2) \cdots (-B^{a_n} + \epsilon_n)$.

Since $\epsilon_1 = \epsilon_3 = 1$ and $\epsilon_2 = -1$, the  monomials of degree greater than $-3$ cancel. Let us see the degree $-3$. If $a_1+1+a_4+1<-3$ it turns out that $F+S = \epsilon_4 \ldots \epsilon_nB^{-3}$; otherwise $a_1+1+a_4+1=-3$, equivalently $a_4=-7$ hence $\epsilon_4=-1$, and
$$
F+S 
= (1+k) \epsilon_2 \epsilon_3 \epsilon_5\cdots \epsilon_n B^{-3} + 
\epsilon_4\cdots \epsilon_n B^{-3}
= (2+k) \epsilon_4\cdots \epsilon_n B^{-3}
$$
since $\epsilon_2\epsilon_3=-1=\epsilon_4$. In both cases $h=-3$ since $k$ is a non-negative integer. 

This completes the proof of the formula in item {\it 5.4}. For example, the span of the Jones polynomial of the pretzel link represented by the pretzel diagram $P(2,-3,-4,-7)$ is $(2 +  3+ 4 + 7) - 6 = 10$. Note that, by applying the formula $\sum_{|a_i|>1} |a_i| - a_1 -3$ of item {\it 5.4} it would be $11$. This amendments the case (iv)(c) in Theorem 2 of \cite{Pedro}. 

\begin{itemize}
\item[{\it 5.5.}] $h = h_F - |a_3| = -1$ if $a_1 = |a_2|-1$, $|a_2| = |a_3|-1$ and $|a_3| = |a_4|-1$, and
\end{itemize}
$\bullet$ Suppose that $a_1 = |a_2|-1$, $|a_2|=|a_3|-1$ and $|a_3|=|a_4|-1$ hence $a_1+a_2 + 1=0$, $a_1+a_3+2=0$ and $a_1+a_4+3=0$. If in $F$ we take two $B^{a_j+1}$ of the last $n-1$ factors, we obtain degree $< -1$ since  $(a_1+1)+(a_2+1)+(a_3+1) = a_3 + 2 = - a_1 < -1$. If in $S$ we take two $-B^{a_j}$ of the last $n-1$ factors, we obtain degree $<-1$ since $(a_1+1) + a_2 + a_3 = a_3 < -1$. Then
$$
\begin{array}{l}
F = \epsilon_2\cdots \epsilon_n B^{a_1+1} + \epsilon_2\cdots \epsilon_nB^{a_1}
+ \epsilon_2\cdots \epsilon_nB^{a_1-1} + \color{blue}{\epsilon_1\cdots \epsilon_n}\color{black}  \\
\qquad + \epsilon_3\cdots \epsilon_n B^{{a_1+1+a_2+1}\color{brown}{=1}\color{black}} 
+ \epsilon_2\epsilon_4\cdots \epsilon_n B^{{a_1+1+a_3+1}\color{blue}{=0}\color{black}} \\
\qquad \qquad + (1+k) \epsilon_2\epsilon_3\epsilon_5\cdots \epsilon_n B^{{a_1+1+a_4+1}\color{red}{=-1}\color{black}} \\
\qquad  + \epsilon_3\cdots \epsilon_n B^{a_1+a_2+1 \color{blue}{=0}\color{black}} 
+ \epsilon_2\epsilon_4\cdots \epsilon_n B^{a_1 + a_3 + 1 \color{red}{=-1}\color{black}} \\ 
\qquad  + \epsilon_3\cdots \epsilon_n B^{a_1 - 1 + a_2 + 1 \color{red}{=-1}\color{black}} \\
\qquad  + \epsilon_3\cdots \epsilon_n B^{a_1 + 1 + a_2 \color{blue}{=0}\color{black}} 
+ \epsilon_2\epsilon_4\cdots \epsilon_n B^{a_1 + 1 + a_3 \color{red}{=-1}\color{black}} \\
\qquad  + \epsilon_3\cdots \epsilon_n B^{a_1 + a_2 \color{red}{=-1}\color{black}} \\
\qquad  + \epsilon_3\cdots \epsilon_n B^{a_1 + 1 + a_2 -1 \color{red}{=-1}\color{black}} 
+ \tn{ monomials of degree $<-1$} 
\end{array}
$$
if $a_4 = a_5 = \ldots = a_{k+4} > a_{k+5}$, and
$$
\begin{array}{l}
S = -\epsilon_2\cdots \epsilon_n B^{a_1+1} - \epsilon_2\cdots \epsilon_nB^{a_1}
- \epsilon_2\cdots \epsilon_nB^{a_1-1} \\
\qquad + \color{brown}{\epsilon_1\cdots \epsilon_n}B  + \color{blue}{\epsilon_1\cdots \epsilon_n}\color{black} + \color{red}{\epsilon_1\cdots \epsilon_nB^{-1}}\color{black}\\ 
\qquad + \epsilon_3\cdots \epsilon_nB^{a_1+1+a_2\color{blue}{=0}\color{black}} 
+ \epsilon_2\epsilon_4\cdots \epsilon_nB^{a_1+1+a_3\color{red}{=-1}\color{black}} \\
\qquad + \epsilon_3\cdots \epsilon_nB^{a_1+a_2\color{red}{=-1}\color{black}} 
+ \tn{ monomials of degree $<-1$}. 
\end{array}
$$
Since $\epsilon_1 = -\epsilon_2 = \epsilon_3$ there is no terms of degree greater that $-1$ and 
$F+S = (1 + k) \epsilon_2\epsilon_3\epsilon_5\cdots \epsilon_n B^{-1} 
+ \tn{ monomials of degree $<-1$}$, with $k \geq 0$. Hence 
$h = -1 = h_F - |a_3|$.

\begin{itemize}
\item[{\it 5.6.}] $h = h_F - |a_2| = 0$ if $a_1 = |a_2| - 1$, $|a_2| = |a_3| - 1$ and $|a_3| = |a_4|$.
\end{itemize}
$\bullet$ Suppose that $a_1 = |a_2|-1$, $|a_2|=|a_3|-1$ and $|a_3|=|a_4|$ hence $a_1+a_2+1 = 0$, $a_1+a_3+2=0$ and $a_1+a_4+2=0$. 
If in $F$ we take two $B^{a_j+1}$ of the last $n-1$ factors, we obtain degree $<0$ since $(a_1+1)+(a_2+1)+(a_3+1) = a_3 + 2 = -a_1 \leq -2$. If in $S$ we take two $-B^{a_j}$ of the last $n-1$ factors, we obtain degree $<0$ since $(a_1+1) + a_2 + a_3 = a_3 < 0$. Then
$$
\begin{array}{l}
F = \epsilon_2\cdots \epsilon_n B^{a_1+1} + \epsilon_2\cdots \epsilon_nB^{a_1}
+ \epsilon_2\cdots \epsilon_nB^{a_1-1} + \color{red}{\epsilon_1\cdots \epsilon_n}\color{black}  \\
\qquad + \epsilon_3\cdots \epsilon_n B^{{a_1+1+a_2+1}\color{blue}{=1}\color{black}} 
+ \epsilon_2\epsilon_4\cdots \epsilon_n B^{{a_1+1+a_3+1}\color{red}{=0}\color{black}} \\
\qquad \qquad + (1+k) \epsilon_2\epsilon_3\epsilon_5\cdots \epsilon_n B^{{a_1+1+a_4+1}\color{red}{=0}\color{black}} \\
\qquad  + \epsilon_3\cdots \epsilon_n B^{a_1+a_2+1 \color{red}{=0}\color{black}} \\
\qquad  + \epsilon_3\cdots \epsilon_n B^{a_1 + 1 + a_2 \color{red}{=0}\color{black}} + \tn{ monomials of degree $<0$}
\end{array}
$$
if $a_4 = a_5 = \ldots = a_{k+4} > a_{k+5}$, and
$$
\begin{array}{l}
S = -\epsilon_2\cdots \epsilon_n B^{a_1+1} - \epsilon_2\cdots \epsilon_nB^{a_1}
- \epsilon_2\cdots \epsilon_nB^{a_1-1} + \color{blue}{\epsilon_1\cdots \epsilon_n}B  + \color{red}{\epsilon_1\cdots \epsilon_n}\color{black}\\ 
\qquad + \epsilon_3\cdots \epsilon_nB^{a_1+1+a_2\color{red}{=0}\color{black}} + \tn{ monomials of degree $<0$}. 
\end{array}
$$
It follows that $F+S = (1 + k)\epsilon_2\epsilon_3\epsilon_5\cdots \epsilon_n + \tn{ monomials of degree $<0$}$ since $\epsilon_1= -\epsilon_2 = \epsilon_3$. In particular $h = 0 = h_F - |a_2|$. This proves item {\it 5.6} and completes the proof of Theorem \ref{TheoremSpan}.
\end{proof}


\begin{center}
\begin{tabular}{c}
Raquel D\'{\i}az\\
Department of Algebra, Geometry and Topology \\
Facultad de Matem\'aticas, Universidad Complutense de Madrid \\
 {\it radiaz@ucm.es} 
\end{tabular}

\begin{tabular}{c}
Pedro M. G. Manch\'on\\
Department of Applied Mathematics to Industrial Engineering \\
ETSIDI, Universidad Polit\'ecnica de Madrid \\
{\it pedro.gmanchon@upm.es} \\
\end{tabular}
\end{center}


\begin{thebibliography}{9}
\bibitem{Clara} Asensio-Mart\'{\i}nez, C.: C\'alculo autom\'atico del polinomio de Jones, Degree Final Work, 2018.
\bibitem{Peter}  Cromwell, P. R.: Knots and links. Cambridge University Press., 2004.
\bibitem{DIP} Dabkowski, M. K., Ishiwata, M. and Przytycki, J. H.: $5$-move equivalence classes of links and their algebraic invariants. J. Knot Theory Ramifications 16 (10) (2007) 1413--1449. 
\bibitem{Landvoy}  Landvoy, R. A.: The Jones polynomial of pretzel knots and links. Topology and its applications 83 (1998) 135--147.
\bibitem{Lickorish}  Lickorish, W. B. R.: An introduction to Knot Theory.
Graduate texts in Mathematics, 175. Springer-Verlag (1997).
\bibitem{Pedro}  Manch\'on, P. M. G.: Kauffman bracket of pretzel links. Marie Curie Fellowships Annals, Second Volume. 118--122 (2003).
\bibitem{Stoimenow} Stoimenow, A.: $5$-moves and Montesinos links. J. Math. Soc. Japan 59, no. 3 (2007), 729--749. 
\bibitem{KnotAtlas} The Knot Atlas, http://katlas.org/
\end{thebibliography}
\end{document}